\documentclass{article}

\usepackage[
backend=biber,
style=ieee,
sorting=none
]{biblatex}
\usepackage[margin=1in]{geometry}
\usepackage[utf8]{inputenc}
\usepackage{csquotes}
\usepackage{amsfonts,amsthm,amsmath,amssymb,accents,mathtools,physics} 
\usepackage{mathrsfs,amsbsy} 
\usepackage{enumitem}
\usepackage{subfig}
\usepackage{graphicx}
\usepackage{float}
\usepackage{xcolor}
\usepackage{algorithm}
\usepackage[noend]{algpseudocode}
\usepackage{hyperref}
\usepackage{authblk} 
\usepackage{textcomp} 
\usepackage{algorithm}
\usepackage{algorithmicx}

\usepackage{bbm}
\usepackage{dsfont}

\numberwithin{equation}{section}
\theoremstyle{definition} 
\newtheorem{rem}{Remark}
\newtheorem{thm}{Theorem}

\newcommand{\bfU}{\mathbf{U}}
\newcommand{\bfV}{\mathbf{V}}
\newcommand{\bfS}{\mathbf{S}}
\newcommand{\bfF}{\mathbf{F}}

\newcommand{\bbP}{\mathbf{\Phi}}
\newcommand{\bfK}{\mathbf{K}}
\newcommand{\bfL}{\mathbf{L}}

\hypersetup{
    colorlinks=true,
    linkcolor=blue,
    citecolor=blue,
    filecolor=blue,      
    urlcolor=blue,
}

\title{Reduced Augmentation Implicit Low-rank (RAIL) integrators for advection-diffusion and Fokker-Planck models}
\author[1]{Joseph Nakao}
\author[2]{Jing-Mei Qiu}
\author[3]{Lukas Einkemmer}
\affil[1]{Department of Mathematics and Statistics, Swarthmore College, Swarthmore, PA, USA}
\affil[2]{Department of Mathematical Sciences, University of Delaware, Newark, DE, USA}
\affil[3]{Department of Mathematics, University of Innsbruck, Innsbruck, Tyrol, Austria}
\date{}

\begin{document}

\maketitle


\begin{abstract}
\noindent This paper introduces a novel computational approach termed the Reduced Augmentation Implicit Low-rank (RAIL) method by investigating two predominant research directions in low-rank solutions to time-dependent partial differential equations (PDEs): dynamical low-rank (DLR), and step-and-truncation (SAT) tensor methods. The RAIL method is designed to enhance the efficiency of traditional full-rank implicit solvers, while maintaining accuracy and stability. We consider spectral methods for spatial discretization, and diagonally implicit Runge-Kutta (DIRK) and implicit-explicit (IMEX) RK methods for time discretization. The efficiency gain is achieved by investigating low-rank structures within solutions at each RK stage. In particular, we develop a reduced augmentation procedure to predict the basis functions to construct projection subspaces. This procedure balances algorithm efficiency and accuracy by incorporating as many bases as possible from previous RK stages, and by optimizing the basis representation through a singular value decomposition (SVD) truncation. As such, one can form implicit schemes for updating basis functions in a dimension-by-dimension manner, similar in spirit to the K-L step in the DLR framework. We propose applying a post-processing step to maintain global mass conservation. We validate the RAIL method through numerical simulations of advection-diffusion problems and a Fokker-Planck model. Our approach generalizes and bridges the DLR and SAT approaches, offering a comprehensive framework for efficiently and accurately solving time-dependent PDEs in the low-rank format with implicit treatment.
\end{abstract}



\section{Introduction}
\label{sec:introduction}

In the realm of computational fluid dynamics, traditional Eulerian mesh-based schemes have been a cornerstone for computing numerical solutions of time-dependent partial differential equations (PDEs). These high-order schemes begin with the spatial discretization of the continuous problem, coupled with different types of temporal discretizations tailored to the needs of the applications at hand. Recently, many researchers have assumed low-rank structures in high-dimensional PDE solutions in order to speed up computational runtimes while maintaining the original high-order accuracy, stability, and robustness of traditional methods. There have been two main lines of research for approximating the solution of time-dependent PDEs by a low-rank decomposition: the \textit{dynamical low-rank (DLR)} approach, and the \textit{step-and-truncation (SAT)} approach. Both the DLR and SAT methods have been used in a number of applications such as quantum mechanics \cite{Lubich2015a,Meyer2009}, plasma physics \cite{Kormann2015,Einkemmer2018,Einkemmer2020,AllmannRahn2022,Einkemmer2023}, radiative transfer transport \cite{PengDLR2020spherical,Einkemmer2021d,Coughlin2022,Kusch2022}, and biology \cite{Jahnke2008,Kusch2021a,Prugger2023,Einkemmer2023a}. These two approaches have a key difference in the order of performing the temporal discretization and low-rank projection in their scheme designs. In most of the literature, the DLR approach performs a low-rank projection before the temporal discretization. Whereas, the SAT approach performs the spatial and temporal discretizations before a low-rank projection.

The DLR methodology dates back to early work in quantum mechanics, e.g., \cite{Meyer1990,Lubich2008}. The overarching idea is to derive a set of differential equations for the low-rank factors by projecting the update onto the tangent space \cite{Koch2007}. The differential equations for the low-rank factors can then be discretized in time using an arbitrary explicit or implicit method in order to obtain a numerical solution. In the original formulation, the resulting differential equations are ill-conditioned in many situations and require regularization \cite{Kieri2016,Nonnenmacher2008}; this has only been addressed recently. The projector splitting integrator \cite{Lubich2014} and the basis updating and Galerkin (BUG) integrator (previously known as the unconventional integrator) \cite{Ceruti2022} avoid this ill-conditioning and attain schemes that are robust with respect to small singular values in the solution. The projector splitting integrator can be raised to higher order using composition (see, e.g., \cite{Einkemmer2018,Cassini2021}) at the price of significantly increased computational cost. However, a robust error bound for the projector splitting integrator only exists for order one. Whereas, the BUG integrator was originally first-order accurate in time and has recently been extended to second-order accuracy \cite{Ceruti2024}. A similar method in which all the differential equations for the bases and Galerkin coefficients are solved in parallel was recently presented in \cite{ceruti2024parallel} and extended to second-order accuracy \cite{Kusch2024}. Retraction-based DLR integrators have also been proposed for their ability to ensure the numerical solution remains on the low-rank manifold \cite{charous2023dynamically,kieri2019projection,seguin2024}; the projected exponential methods in \cite{Carrel2023} are very similar. Notably, the projected Runge-Kutta methods in \cite{kieri2019projection} have robust higher-order error bounds.

SAT methods offer an alternative approach in which the low-rank solution is updated over a single time-step and then truncated. Depending on the specific equation used, stepping a low-rank approximation forward in time with a time integrator increases the rank of the solution. Therefore, this step is followed by a truncation, usually using the singular value decomposition (SVD), that removes the singular values that are sufficiently small. The SAT approach has been applied to high-dimensional time-dependent problems fairly recently, particularly for kinetic models with earlier works in \cite{Dolgov2014,Ehrlacher2017,Kormann2015}. In the past few years, this approach has seen significant research efforts in the context of explicit schemes, e.g., \cite{GuoVlasovFlowMap2022, rodgers2022adaptive, guo2024conservative}. The approach, however, is not easily applicable to implicit methods, as seeking the low-rank representation of a solution in an implicit setting with efficiency and accuracy is more challenging.

Recently, implicit low-rank methods have seen increased interest in the literature. Rodgers and Venturi introduced an implicit approach by directly solving algebraic equations on tensor manifolds \cite{rodgers2022implicit}. El Kahza et al. \cite{ElKahza2024} recently presented a high-order rank-adaptive implicit integrator based on leveraging extended Krylov subspaces to efficiently populate the low-rank bases. Appelo and Cheng proposed a rank-adaptive scheme for handling equations with cross terms in which the bases are predicted explicitly and then used to implicitly solve a reduced matrix system \cite{Appelo2024}. Naderi et al. \cite{naderi2024cur} introduced a cost-effective Newton's method for solving stiff, nonlinear matrix differential equations arising from discretizations of random PDEs on low-rank matrix manifolds using high-order implicit multistep and Runge-Kutta (RK) methods.

We follow a similar framework to the DLR and SAT schemes, developing efficient low-rank solvers with implicit and implicit-explicit (IMEX) temporal discretizations of the PDE. The key algorithm design question for schemes that involve implicit time discretizations lies in effectively seeking basis functions for each dimension at future times. In this paper, we propose an efficient Reduced Augmentation Implicit Low-rank (RAIL) method for time-dependent PDEs. To achieve this, we first fully discretize the PDE in space using a spectral method \cite{hesthaven2007spectral,trefethen2000spectral}, and in time using diagonally implicit Runge-Kutta (DIRK) or IMEX Runge-Kutta methods as in the SAT approach. To update the low-rank factorization of the solution in an implicit fashion, we then follow a very similar approach to the augmented BUG integrator \cite{Ceruti2022a} to perform the basis update step. The distinctive feature in this work is leveraging the basis update and Galerkin projection procedure from \cite{Ceruti2022a} to look for low-rank factorizations of solutions {\em at each RK stage}. This is similar in spirit to the projected RK methods in \cite{kieri2019projection} in which the dynamics are projected at each stage. In order to balance efficiency and accuracy, we propose a reduced augmentation procedure to predict the bases for dimension-wise subspaces to be used in the projection step. The reduced augmentation procedure spans the bases from a prediction step together with those from all previous RK stages to construct richer bases. The procedure also has a follow-up SVD truncation to optimize efficiency in constructing these basis representations. This reduced augmentation procedure is designed with efficiency in mind (to truncate redundancy in the basis representations), while attempting to balance the accuracy of RK methods (to accommodate as many bases as possible from previous RK approximations). 

This ordering of first \textit{discretizing in time} and then performing a \textit{low-rank truncation} leverages ideas from both the DLR and SAT methodologies. In fact, for the implicit backward Euler scheme, we obtain the same numerical solution as the augmented BUG integrator. Thus, the proposed approach can be seen as a generalization of that method. We analyze the first-order RAIL scheme, and the accuracy of the higher-order scheme is shown through numerical experiments. We demonstrate the utility of our proposed method by performing numerical simulations for advection-diffusion problems and a Fokker-Planck model. We also demonstrate how it can be combined using the recently developed conservative projection procedure \cite{Einkemmer2021a,Einkemmer2022,guo2024conservative} to obtain a scheme that is globally mass conservative.

The remainder of the paper is organized as follows. Section \ref{sec:overall} formulates the matrix differential equation arising from discretizing the PDE. Section \ref{sec:implicit} presents the proposed RAIL scheme using implicit RK methods. Section \ref{sec:RAIL_IMEX} extends the proposed RAIL scheme to IMEX Runge-Kutta time discretizations. Section \ref{sec:numerics} presents numerical examples demonstrating the effectiveness and efficiency of the proposed method. Finally, a brief summary of the proposed work and the plan for future work is described in Section \ref{sec:conclusion}. The appendices includes: Butcher tableaus for the DIRK methods, a description of how we spatially discretized the advection-diffusion equations solved in this paper, and the summarized algorithms for the RAIL scheme, including a globally mass conservative truncation procedure.

\section{The matrix differential equation arising from PDE discretization in space}
\label{sec:overall}

To best illustrate the effectiveness of our proposed implicit and implicit-explicit low-rank algorithms, we consider advection-diffusion equations of the form
\begin{equation}
    \label{eq:adv_diff}
    \begin{cases}
    u_t + \nabla\cdot\left(\mathbf{a}(\mathbf{x},t)u\right) = \nabla\cdot(\mathbf{D}\nabla u) + \phi(\mathbf{x},t),&\mathbf{x}\in\Omega,\quad t>0,\\
    u(\mathbf{x},t=0)=u_0(\mathbf{x}),&\mathbf{x}\in\Omega,
    \end{cases}
\end{equation}
where $\mathbf{a}(\mathbf{x},t)$ describes the flow field, $\mathbf{D}$ is the anisotropic diffusion tensor, and  $\phi(\mathbf{x},t)$ is the source term. For context, several Fokker-Planck type kinetic models can be cast as an advection-diffusion equation of the form \eqref{eq:adv_diff}. Since such kinetic models typically possess low-rank solutions and are generally high-dimensional, equation \eqref{eq:adv_diff} acts as a good model to investigate storage and computational savings under a low-rank assumption.

Our proposed low-rank methodology relies on the Schmidt decomposition of two-dimensional solutions. The solution is decomposed as a linear combination of orthonormal basis functions, 
\begin{equation}
    \label{eq:schmidt}
    u(x,y,t) = \sum\limits_{i=1}^{r}{s_{i}(t) V_i^x(x,t)V_i^y(y,t)},
\end{equation}
where $r$ is the rank, $\{V_i^x(x,t)\text{ : }i=1,2,...,r\}$ and $\{V_i^y(y,t)\text{ : }i=1,2,...,r\}$ are orthonormal time-dependent bases in the $x$ and $y$ dimensions, respectively, and $\{s_{i}(t)\text{ : }i=1,2,..,r\}$ are the corresponding time-dependent coefficients. These basis functions are typically ordered by the magnitude of their coefficients $s_i(t)$ to allow for truncation based on some tolerance. At the discrete level, the Schmidt decomposition \eqref{eq:schmidt} can be interpreted as the SVD for the solution matrix on a tensor product of uniform computational grids, $\bf{X} \otimes \bf{Y}$, with 
\begin{equation}
\label{eq: 2dgrid}
\mathbf{X}:    
x_1 < \cdots < x_i <\cdots < x_{N}, \qquad
\mathbf{Y}:   
y_1 < \cdots < y_j <\cdots < y_{N}. 
\end{equation}

When low-rank structure exists in the solution, it can be exploited to reduce the computational complexity. The discrete analogue of the Schmidt decomposition \eqref{eq:schmidt} is the SVD
\begin{equation}
\label{eq:matrixsoln}
\mathbf{U}(t)  = \mathbf{V}^x(t) \mathbf{S}(t)(\mathbf{V}^y(t))^T,
\end{equation}
where $\mathbf{V}^x(t) \in \mathbb{R}^{N \times r}$ and $\mathbf{V}^y(t) \in \mathbb{R}^{N \times r}$ have orthonormal columns, and $\mathbf{S}(t)\in\mathbb{R}^{r\times r}$ is diagonal with entries organized in decreasing magnitude. It is important to note that in general the solution \eqref{eq:matrixsoln} is not necessarily the SVD; the SVD is a special case.

In this paper, we develop a low-rank approach that incorporates implicit time discretizations to deal with stiff terms in PDE models. Taking equation \eqref{eq:adv_diff} as an example, we will first discretize in space to get the matrix differential equation
\begin{equation}
    \label{eq:MOL}
    \frac{d\bfU}{d t} = \textbf{Ex}(t,\bfU) +  \textbf{Im}(t,\bfU)
   + \bbP(t). 
\end{equation}
Here $\bfU$ is the matrix solution of the PDE, and $\bbP$ is the matrix value of any possible source term on a 2D tensor product grid \eqref{eq: 2dgrid}. $\textbf{Ex}(t,\bfU)$ represents nonstiff operators treated explicitly in time, e.g., the advection terms in \eqref{eq:adv_diff}. $\textbf{Im}(t,\bfU)$ represents stiff operators treated implicitly in time, e.g., diffusion operators. Depending on the stiffness of the source term, $\mathbf{\Phi}$ could be treated explicitly or implicitly. At the matrix level, operators involving first and second derivatives in $x$ and $y$ can be well represented by matrix multiplication from the left and right, respectively, acting on the basis functions in the corresponding $x$ and $y$ directions. Appendix \ref{app:advdiffdisc} presents how we spatially discretize advection-diffusion equations with low-rank flow fields and diagonal diffusion tensors.

For treating explicit terms only, the ``step-and-truncate" approach in \cite{GuoVlasovFlowMap2022} first adds discretizations of the righthand side in the low-rank format ({\em step}), followed by truncating the updated solution based on some tolerance criteria ({\em truncate}). However, when implicit treatment is involved, seeking low-rank representations of solutions \eqref{eq:matrixsoln} is more challenging since knowledge of the future bases is required. This motivates the need to construct bases that can accommodate solutions for $t\in[t^n,t^{n+1}]$.

\section{The implicit RAIL scheme}
\label{sec:implicit}
We start by considering implicit discretizations of systems that can be expressed in the form
\begin{equation}
    \label{eq:MOL_implicit}
    \frac{d\bfU}{d t} = \bfF_{x} \bfU + \bfU \bfF_{y}^T.
\end{equation}

For simplicity, we consider the diffusion equation for which $\bfF_{x}$ and $\bfF_{y}$ are constant and discretize the Laplacian. However, more general models can be considered for which $\mathbf{F}_x$ depends on $x$ and $t$, and $\mathbf{F}_y$ depends on $y$ and $t$, e.g., Fokker-Planck type kinetic models. Throughout this paper, we only consider stiffly accurate DIRK methods for the time discretization \cite{Hairer1996}. The proposed method can be extended to other implicit time discretizations, e.g., linear multi-step methods; preliminary numerical results are shown in \cite{nakao2023thesis}. We start by fully discretizing the matrix differential equation \eqref{eq:MOL_implicit} in time. Let $\bfU^n\approx\bfU(t^n)$, or rather,
\[\bfV^{x,n}\bfS^{n}(\bfV^{y,n})^T\approx\bfV^{x}(t^n)\bfS(t^n)(\bfV^{y}(t^n))^T,\]
where $t^n=n\Delta t$ denotes the time after $n$ time-steps. Although the time-stepping size $\Delta t$ could vary at each time-step, we assume a constant $\Delta t$ for simplicity.

\subsection{The first-order implicit scheme}\label{sec:implicit_firstorderscheme}

Discretizing equation \eqref{eq:MOL_implicit} in time using backward Euler method, the resulting matrix equation is
\begin{equation}
	\label{eq:bEuler}
	\bfU^{n+1} = \bfU^{n} + \Delta t\Big(\bfF_{x} \bfU^{n+1} + \bfU^{n+1} \bfF_{y}^T\Big),
\end{equation}
where $\bfU^{n+1}=\bfV^{x,n+1}\bfS^{n+1}(\bfV^{y,n+1})^T$.

Following the spirit of DLR approach, we propose updating the bases dimension-by-dimension by projecting equation \eqref{eq:bEuler} onto low-dimensional subspaces spanned by $\bfV^{x,n+1}_{\star}$ and $\bfV^{y,n+1}_{\star}$. We define $\bfV^{x,n+1}_{\star}\coloneqq\bfV^{x,n}$ and $\bfV^{y,n+1}_{\star}\coloneqq\bfV^{y,n}$ since the bases at $t^n$ are all the information we currently know. By projecting the solution and matrix equation in one spatial dimension, we can then update the basis vectors one dimension at a time. In the DLR literature, these are commonly known as the K and L steps, in which $\mathbf{K}$ and $\mathbf{L}$ (or $K$ and $L$ if the projection is performed at the continuous level before discretizing in space) respectively denote the solution projected in $y$ and $x$. DLR methods (including the BUG integrators) then solve the resulting matrix differential equations for $\mathbf{K}^{n+1}$ and $\mathbf{L}^{n+1}$. In our proposed method, we perform the orthogonal projection after the time discretization. For the proposed first-order scheme in which $\mathbf{V}^{\cdot,n+1}_{\star}\coloneqq\mathbf{V}^{\cdot,n}$, we let
\begin{subequations}
\begin{align}
\begin{split}\label{eq:K_bE}
\bfK^{n+1} &\coloneqq \bfV^{x,n+1}\bfS^{n+1}(\bfV^{y,n+1})^T \bfV^{y,n}\in\mathbb{R}^{N\times r^{n}},
\end{split}
\\
\begin{split}\label{eq:L_bE}
\bfL^{n+1} &\coloneqq \bfV^{y,n+1}(\mathbf{S}^{n+1})^T(\bfV^{x,n+1})^T\bfV^{x,n}\in\mathbb{R}^{N\times r^{n}},
\end{split}
\end{align}
\end{subequations}
where $r^n<N$ is the rank of the solution at $t^n$, and we have taken the transpose of the solution $\mathbf{U}^{n+1}$ when defining $\mathbf{L}^{n+1}$ for convenience. Without loss of generality, we only show work for the K step since the L step follows the same procedure. Projecting equation \eqref{eq:bEuler} in $y$ using $\bfV^{y,n+1}_{\star}=\bfV^{y,n}$,
\begin{equation}\label{eq:Kstep1_bE}
\mathbf{K}^{n+1} = \Bigg[\mathbf{U}^n + \Delta t\Big(\bfF_{x} \bfU^{n+1} + \bfU^{n+1} \bfF_{y}^T\Big)\Bigg]\bfV^{y,n}.
\end{equation}

A natural way to solve equation \eqref{eq:Kstep1_bE} is to rewrite it as a Sylvester equation. Looking at equation \eqref{eq:Kstep1_bE}, we clearly see that $\mathbf{F}_x\bfU^{n+1}\bfV^{y,n}=\mathbf{F}_x\mathbf{K}^{n+1}$. However, we must project the solution $\mathbf{U}^{n+1}$ in the $\mathbf{F}_y$ term to get a Sylvester form. Approximating this term with the projected solution $\mathbf{U}^{n+1}\bfV^{y,n}(\bfV^{y,n})^T\mathbf{F}_y^T\bfV^{y,n}$,
\begin{equation}\label{eq:Kstep2_bE}
\mathbf{K}^{n+1} = \mathbf{U}^n\bfV^{y,n} + \Delta t\Big(\bfF_{x}\mathbf{K}^{n+1} + \mathbf{K}^{n+1}(\bfF_{y}\bfV^{y,n})^T\bfV^{y,n}\Big).
\end{equation}

As such, the K equation to solve is the Sylvester equation
\begin{equation}\label{eq:Kstep3_bE}
\Big(\mathbf{I}-\Delta t\mathbf{F}_x\Big)\mathbf{K}^{n+1} - \mathbf{K}^{n+1}\Big(\Delta t(\mathbf{F}_y\bfV^{y,n})^T\bfV^{y,n}\Big) = \bfV^{x,n}\bfS^{n}.
\end{equation}

Similarly, the L equation to solve is the Sylvester equation
\begin{equation}\label{eq:Lstep3_bE}
\Big(\mathbf{I}-\Delta t\mathbf{F}_y\Big)\mathbf{L}^{n+1} - \mathbf{L}^{n+1}\Big(\Delta t(\mathbf{F}_x\bfV^{x,n})^T\bfV^{x,n}\Big) = \bfV^{y,n}(\bfS^{n})^T.
\end{equation}

Note that equations \eqref{eq:Kstep3_bE} and \eqref{eq:Lstep3_bE} can be solved in parallel. Referring back to how we defined $\mathbf{K}^{n+1}$ in \eqref{eq:K_bE}, the desired orthonormal basis in $x$ can be computed via a reduced QR factorization, $\mathbf{K}^{n+1}=\mathbf{QR}\eqqcolon\bfV^{x,n+1}_{\ddagger}\mathbf{R}$. Similarly, $\mathbf{L}^{n+1}=\mathbf{QR}\eqqcolon\bfV^{y,n+1}_{\ddagger}\mathbf{R}$. Here the double dagger $\ddagger$ denotes the orthonormal bases obtained from the $\mathbf{K}$ and $\mathbf{L}$ equations predicting the solution basis at $t^{n+1}$. We toss the upper triangular $\mathbf{R}$ from the QR factorizations since we only need an orthonormal basis.

With $\bfV^{\cdot,n+1}_{\ddagger}$ computed, we propose augmenting it with $\bfV^{\cdot,n}$. By containing bases at times $t^{n+1}$ and $t^n$, the updated basis spans an enriched subspace that approximates the basis $\bfV^{\cdot}(t)$ over the interval $[t^n,t^{n+1}]$. This is particularly important when the solution rapidly evolves over short time intervals. We note that one could equivalently augment $\bfV^{x,n}$ with $\mathbf{K}^{n+1}$ (and augment $\bfV^{y,n}$ with $\mathbf{L}^{n+1}$). This augmentation procedure for the K and L steps was proposed in the augmented BUG integrator \cite{Ceruti2022a}. Up until this point, the proposed first-order scheme has been identical to the augmented BUG integrator assuming backward Euler time discretization.

Here is where the proposed first-order scheme differs from the augmented BUG integrator. The augmented basis $[\bfV^{\cdot,n+1}_{\ddagger},\bfV^{\cdot,n}]$ is size $N\times (r^{n+1}_{\ddagger}+r^n)$, that is, roughly $N\times 2r^n$. Although rank $2r^n$ is not too large, the augmented bases could have several redundancies, particularly if the solution is not changing rapidly. This could dramatically increase the computational cost if we eventually want to consider augmenting more bases. To help improve the computational efficiency, we propose a \textit{reduced augmentation} procedure to determine the updated bases. Simply put, we truncate the augmented bases according to a small tolerance of $10^{-12}$. This tolerance is large enough to significantly reduce the rank, but small enough to not dramatically affect the scheme. Computing the reduced QR factorizations of the augmented bases,
\begin{subequations}\label{eq:redaug}
\begin{align}
\begin{split}
\Big[\ \bfV^{x,n+1}_{\ddagger}\ ,\ \bfV^{x,n}\ \Big] &= \mathbf{Q}_x\mathbf{R}_x,
\end{split}
\\
\begin{split}
\Big[\ \bfV^{y,n+1}_{\ddagger}\ ,\ \bfV^{y,n}\ \Big] &= \mathbf{Q}_y\mathbf{R}_y.
\end{split}
\end{align}
\end{subequations}
Let $\hat{r}^{n+1}$ be the maximum of the number of singular values of $\mathbf{R}_x$ and $\mathbf{R}_y$ larger than $10^{-12}$. Further let $\hat{\bfV}^{x,n+1}$ be $\mathbf{Q}_x$ multiplied on the right by the first $\hat{r}^{n+1}$ left singular vectors of $\mathbf{R}_x$; let $\hat{\bfV}^{y,n+1}$ be $\mathbf{Q}_y$ multiplied on the right by the first $\hat{r}^{n+1}$ left singular vectors of $\mathbf{R}_y$.

With the updated bases $\hat{\bfV}^{x,n+1}$ and $\hat{\bfV}^{y,n+1}$ computed via reduced augmentation, we can now use them to perform a Galerkin projection to update the coefficients held in $\hat{\mathbf{S}}^{n+1}$; this is commonly known as the S step in the DLR literature. Here we have let hats denote the factorized solution produced by the KLS procedure. Projecting equation \eqref{eq:bEuler} in both dimensions by multiplying on the left by $(\hat{\bfV}^{x,n+1})^T$ and on the right by $\hat{\bfV}^{y,n+1}$, we get the Sylvester equation
\begin{align}\label{eq:Sstep_bE}
\begin{split}
\Big(\mathbf{I} - \Delta t(\hat{\bfV}^{x,n+1})^T\mathbf{F}_x\hat{\bfV}^{x,n+1}\Big)\hat{\mathbf{S}}^{n+1} - &\hat{\mathbf{S}}^{n+1}\Big(\Delta t(\mathbf{F}_y\hat{\bfV}^{y,n+1})^T\hat{\bfV}^{y,n+1}\Big)\\
&= (\hat{\bfV}^{x,n+1})^T\bfV^{x,n}\bfS^n(\bfV^{y,n})^T\hat{\bfV}^{y,n+1}.
\end{split}
\end{align}

Similar to the augmented BUG integrator, as well as the SAT framework, we perform one final (SVD-type) truncation on the updated solution $\hat{\bfV}^{x,n+1}\hat{\bfS}^{n+1}(\hat{\bfV}^{y,n+1})^T$ to optimize computational efficiency. The simplest strategy is to compute the SVD $\hat{\bfS}^{n+1} = \mathbf{A}\mathbf{\Sigma}\mathbf{B}^T$ and truncate the singular values according to some tolerance $\epsilon$. Letting $r^{n+1}\leq\hat{r}^{n+1}$ be the number of singular values larger than $\epsilon$, the final solution is $\bfU^{n+1} = \bfV^{x,n+1}\bfS^{n+1}(\bfV^{y,n+1})^T$, where $\bfV^{x,n+1}=\hat{\bfV}^{x,n+1}\mathbf{A}(\ :\ ,1:r^{n+1})$, $\bfS^{n+1}=\hat{\bfS}^{n+1}(1:r^{n+1},1:r^{n+1})$, and $\bfV^{y,n+1}=\hat{\bfV}^{y,n+1}\mathbf{B}(\ :\ ,1:r^{n+1})$.


\subsection{The second-order implicit scheme}\label{sec:implicit_secondorderscheme}

The novelty and advantage of the reduced augmentation that distinguish it from the augmented BUG integrator \cite{Ceruti2022a} are best observed in the higher-order extension. With the first-order implicit scheme established, we extend to the following two-stage second-order DIRK method \cite{Ascher1997,Hairer1996}:
\begin{subequations}
\begin{align}\label{eq:DIRK2.1}
\begin{split}
	\bfU^{(1)} &= \bfU^n + \nu\Delta t\Big(\mathbf{F}_x\bfU^{(1)}+\bfU^{(1)}\mathbf{F}_y^T\Big),
\end{split}\\
\begin{split}\label{eq:DIRK2.2}
	\bfU^{n+1} &= \bfU^n + (1-\nu)\Delta t\Big(\mathbf{F}_x\bfU^{(1)}+\bfU^{(1)}\mathbf{F}_y^T\Big) + \nu\Delta t\Big(\mathbf{F}_x\bfU^{n+1}+\bfU^{n+1}\mathbf{F}_y^T\Big),
\end{split}
\end{align}
\end{subequations}
where $\nu=1-\sqrt{2}/2$ and $\bfU^{(1)}=\bfV^{x,(1)}\bfS^{(1)}(\bfV^{y,(1)})^T$ approximates the solution at time $t^{(1)}=t^n+\nu\Delta t$. Since this second-order DIRK method has two stages in equations \eqref{eq:DIRK2.1}-\eqref{eq:DIRK2.2}, we propose performing the KLS(+truncate) process at \text{both} stages. Doing so gives the flexibility to project onto different subspaces at each stage.

Since the first stage in equation \eqref{eq:DIRK2.1} is a backward Euler discretization over a step-size of $\nu\Delta t$, we use the first-order RAIL scheme presented in subsection \ref{sec:implicit_firstorderscheme} to efficiently update the intermediate solution $\bfU^{(1)}$ in factorized form. All that remains is performing the KLS(+truncate) process on equation \eqref{eq:DIRK2.2}. Since equation \eqref{eq:DIRK2.2} balances terms from many stages, we need the basis vectors in $\bfV^{\cdot,n+1}_{\star}$ to span a richer space. In the first-order RAIL scheme, we let $\bfV^{\cdot,n+1}_{\star}\coloneqq\bfV^{\cdot,n}$. But in the current situation, we also have the bases computed at the first RK stage that we want to include/augment in our projection space. Herein lies a major difference between the first- and second-order RAIL schemes: the first-order RAIL scheme only utilized the reduced augmention for the S step, but the second stage of the second-order RAIL scheme utilizes the reduced augmentation for the K-L step \textit{and} the S step. For the K-L step, we construct the approximate bases $\bfV^{\cdot,n+1}_{\star}$ by reducing the augmented matrices
\begin{subequations}\label{eq:augDIRK2}
\begin{align}
\begin{split}
\Big[\ \bfV^{x,n+1}_{\dagger}\ ,\ \bfV^{x,(1)}\ ,\ \bfV^{x,n}\ \Big],
\end{split}
\\
\begin{split}
\Big[\ \bfV^{y,n+1}_{\dagger}\ ,\ \bfV^{y,(1)}\ ,\ \bfV^{y,n}\ \Big],
\end{split}
\end{align}
\end{subequations}
according to tolerance $10^{-12}$, where $\bfV^{\cdot,n+1}_{\dagger}$ is a first-order prediction at time $t^{n+1}$ computed using the first-order RAIL scheme. For the same reason that $\bfV^{\cdot,n+1}_{\ddagger}$ was included in the augmentation for the S step in the first-order RAIL scheme, we include $\bfV^{\cdot,n+1}_{\dagger}$ in the augmentation \eqref{eq:augDIRK2} for the K and L steps since it enriches the space with information at time $t^{n+1}$ and is relatively cheap to compute. Note that reducing the augmented bases \eqref{eq:augDIRK2} is critical for maintaining computational efficiency since the rank would have otherwise been roughly $3r^n$. Projecting the solution $\bfU^{n+1}$ onto $\bfV^{y,n+1}_{\star}$ or $\bfV^{x,n+1}_{\star}$ respectively defines
\begin{subequations}
\begin{align}
\begin{split}\label{eq:K_DIRK2}
\bfK^{n+1} &\coloneqq \bfV^{x,n+1}\bfS^{n+1}(\bfV^{y,n+1})^T \bfV^{y,n+1}_{\star}\in\mathbb{R}^{N\times r^{n+1}_{\star}},
\end{split}
\\
\begin{split}\label{eq:L_DIRK2}
\bfL^{n+1} &\coloneqq \bfV^{y,n+1}(\mathbf{S}^{n+1})^T(\bfV^{x,n+1})^T\bfV^{x,n+1}_{\star}\in\mathbb{R}^{N\times r^{n+1}_{\star}}.
\end{split}
\end{align}
\end{subequations}

Equation \eqref{eq:DIRK2.2} can be rewritten as
\begin{equation}
	\label{eq:DIRK2.2_v2}
	\bfU^{n+1} = \mathbf{W}^{(1)} + \nu\Delta t\Big(\bfF_{x} \bfU^{n+1} + \bfU^{n+1} \bfF_{y}^T\Big),
\end{equation}
where
\begin{equation}
	\label{eq:DIRK2.2_RHS}
	\mathbf{W}^{(1)} = \bfU^{n} + (1-\nu)\Delta t\Big(\bfF_{x} \bfU^{(1)} + \bfU^{(1)} \bfF_{y}^T\Big).
\end{equation}

Projecting equation \eqref{eq:DIRK2.2_v2} onto the subspace spanned by the columns of $\bfV^{y,n+1}_{\star}$, and following the same process that yielded equation \eqref{eq:Kstep3_bE} from \eqref{eq:Kstep1_bE}, we get the Sylvester equation
\begin{equation}\label{eq:Kstep_DIRK2.2}
\Big(\mathbf{I}-\nu\Delta t\mathbf{F}_x\Big)\mathbf{K}^{n+1} - \mathbf{K}^{n+1}\Big(\nu\Delta t(\mathbf{F}_y\bfV^{y,n+1}_{\star})^T\bfV^{y,n+1}_{\star}\Big) = \mathbf{W}^{(1)}\bfV^{y,n+1}_{\star}.
\end{equation}

Similarly, the L equation is the Sylvester equation
\begin{equation}\label{eq:Lstep_DIRK2.2}
\Big(\mathbf{I}-\nu\Delta t\mathbf{F}_y\Big)\mathbf{L}^{n+1} - \mathbf{L}^{n+1}\Big(\nu\Delta t(\mathbf{F}_x\bfV^{x,n+1}_{\star})^T\bfV^{x,n+1}_{\star}\Big) = (\mathbf{W}^{(1)})^T\bfV^{x,n+1}_{\star}.
\end{equation}

After solving equations \eqref{eq:Kstep_DIRK2.2}-\eqref{eq:Lstep_DIRK2.2} for $\mathbf{K}^{n+1}$ and $\mathbf{L}^{n+1}$, reduced QR factorizations yield the orthonormal bases $\bfV^{x,n+1}_{\ddagger}$ and $\bfV^{y,n+1}_{\ddagger}$. Again, double dagger $\ddagger$ denotes the orthonormal bases obtained from the $\mathbf{K}$ and $\mathbf{L}$ equations. Repeating the reduced augmentation procedure on the augmented bases
\begin{subequations}
\begin{align}
\begin{split}
\Big[\ \bfV^{x,n+1}_{\ddagger}\ ,\ \bfV^{x,(1)}\ ,\ \bfV^{x,n}\ \Big],
\end{split}
\\
\begin{split}
\Big[\ \bfV^{y,n+1}_{\ddagger}\ ,\ \bfV^{y,(1)}\ ,\ \bfV^{y,n}\ \Big],
\end{split}
\end{align}
\end{subequations}
yields the pre-truncated bases $\hat{\bfV}^{\cdot,n+1}$ of rank $\hat{r}^{n+1}$. The Sylvester equation resulting from the Galerkin projection onto $\hat{\bfV}^{x,n+1}$ and $\hat{\bfV}^{y,n+1}$ is
\begin{align}\label{eq:Sstep_DIRK2.2}
\begin{split}
\Big(\mathbf{I} - \nu\Delta t(\hat{\bfV}^{x,n+1})^T\mathbf{F}_x\hat{\bfV}^{x,n+1}\Big)\hat{\mathbf{S}}^{n+1} - &\hat{\mathbf{S}}^{n+1}\Big(\nu\Delta t(\mathbf{F}_y\hat{\bfV}^{y,n+1})^T\hat{\bfV}^{y,n+1}\Big)\\
&\qquad\qquad= (\hat{\bfV}^{x,n+1})^T\mathbf{W}^{(1)}\hat{\bfV}^{y,n+1}.
\end{split}
\end{align}

With the coefficients $\hat{\bfS}^{(k)}$ computed, the truncated updated solution $\bfU^{n+1}=\bfV^{x,n+1}\bfS^{n+1}(\bfV^{y,n+1})^T$ can be obtained; see subsection \ref{sec:implicit_firstorderscheme}.


\subsection{High-order implicit schemes}

The extension to high-order DIRK schemes is presented and follows very similarly to the second-order extension in subsection \ref{sec:implicit_secondorderscheme}. RK schemes are commonly expressed by a Butcher tableau. In particular, a general $s$ stage DIRK method is expressed by the following Butcher tableau:
\begin{table}[h!]
\centering
\caption{Butcher tableau for an $s$ stage DIRK scheme}
\label{table:DIRK}
\begin{tabular}{c|llll}
    $c_1$&$a_{11}$&0&$\hdots$&0\\
    $c_2$&$a_{21}$&$a_{22}$&$\hdots$&0\\
    $\vdots$&$\vdots$&$\vdots$&$\ddots$&$\vdots$\\
    $c_s$&$a_{s1}$&$a_{s2}$&$\hdots$&$a_{ss}$\\
    \hline
    &$b_1$&$b_2$&$\hdots$&$b_s$
    \end{tabular}
\end{table}

In Table \ref{table:DIRK}, $c_k = \sum_{j=1}^{k}{a_{kj}}$ for $k=1,2,...,s$, and $\sum_{k=1}^{s}{b_k} = 1$ for consistency. Each row in the Butcher tableau represents an intermediate stage for the solution at time $t^{(k)}=t^n+c_k\Delta t$ for $1,2,...,s$. Referring back to the matrix differential equation \eqref{eq:MOL_implicit}, the intermediate equation at the $k$th stage is
\begin{equation}
	\label{eq:DIRK_k}
	\bfU^{(k)} = \bfU^{n} + \Delta t\sum\limits_{\ell=1}^{k}{a_{k\ell}\mathbf{Y}_{\ell}} \equiv \bfU^{n} + \Delta t\sum\limits_{\ell=1}^{k}{a_{k\ell}\Big(\bfF_{x} \bfU^{(\ell)} + \bfU^{(\ell)} \bfF_{y}^T\Big)},
\end{equation}
where $\bfU^{(k)}=\bfV^{x,(k)}\bfS^{(k)}(\bfV^{y,(k)})^T\approx\bfU(t^{(k)})$. With all the intermediate stages computed, the final solution is
\begin{equation}
	\label{eq:DIRK_nn}
	\bfU^{n+1} = \bfU^n + \Delta t\sum\limits_{k=1}^{s}{b_k\mathbf{Y}_k}.
\end{equation}

We restrict ourselves to stiffly accurate DIRK methods for which $c_s=1$ and $a_{sk}=b_k$ for $k=1,2,...,s$. As such, $\bfU^{(s)}=\bfU^{n+1}$ and we do not need to compute equation \eqref{eq:DIRK_nn}. Note that the second-order DIRK method in subsection \ref{sec:implicit_secondorderscheme} is stiffly accurate. Only equation \eqref{eq:DIRK_k} needs to be solved at each stage. This begs the question: what approximate bases $\bfV^{\cdot,(k)}_{\star}$ should we use in the projection of equation \eqref{eq:DIRK_k} at the $k$th stage? Since the RAIL framework discretizes in time before performing the projection, we have the flexibility to construct richer low-rank subspaces for each subsequent RK stage. Extending the idea from the second-order RAIL scheme, we propose enriching the basis as much as possible by spanning the bases computed at all the previous stages, as well as a first-order prediction at time $t^{(k)}$. Then, we can perform the KLS(+truncate) process at \textit{each} stage on equation \eqref{eq:DIRK_k}. For the K and L steps, we construct $\bfV^{\cdot,(k)}_{\star}$ by reducing the following augmented matrices:
\begin{subequations}\label{eq:aug}
\begin{align}
\begin{split}
\Big[\ \bfV^{x,(k)}_{\dagger}\ ,\ \bfV^{x,(k-1)}\ ,\ ...\ ,\ \bfV^{x,(1)}\ ,\ \bfV^{x,n}\ \Big],
\end{split}
\\
\begin{split}
\Big[\ \bfV^{y,(k)}_{\dagger}\ ,\ \bfV^{y,(k-1)}\ ,\ ...\ ,\ \bfV^{y,(1)}\ ,\ \bfV^{y,n}\ \Big],
\end{split}
\end{align}
\end{subequations}
where $\bfV^{\cdot,(k)}_{\dagger}$ denotes a first-order prediction at time $t^{(k)}$ computed using the first-order RAIL scheme. Same as the second-order RAIL scheme, $\bfV^{\cdot,(k)}_{\dagger}$ is only needed for $k>1$, i.e., not the first stage since it is a backward Euler discretization and we just use the first-order RAIL scheme to compute $\bfU^{(1)}$ in factorized form. Observe that the reduced augmentation maintains computational efficiency by reducing the rank from what would otherwise be roughly $(k+1)r^n$. Projecting the intermediate solution $\bfU^{(k)}$ onto $\bfV^{y,(k)}_{\star}$ or $\bfV^{x,(k)}_{\star}$ respectively defines
\begin{subequations}
\begin{align}
\begin{split}\label{eq:K_DIRK}
\bfK^{(k)} &\coloneqq \bfV^{x,(k)}\bfS^{(k)}(\bfV^{y,(k)})^T \bfV^{y,(k)}_{\star}\in\mathbb{R}^{N\times r^{(k)}_{\star}},
\end{split}
\\
\begin{split}\label{eq:L_DIRK}
\bfL^{(k)} &\coloneqq \bfV^{y,(k)}(\mathbf{S}^{(k)})^T(\bfV^{x,(k)})^T\bfV^{x,(k)}_{\star}\in\mathbb{R}^{N\times r^{(k)}_{\star}}.
\end{split}
\end{align}
\end{subequations}

Equation \eqref{eq:DIRK_k} can be rewritten as
\begin{equation}
	\label{eq:DIRK_k2}
	\bfU^{(k)} = \mathbf{W}^{(k-1)} + a_{kk}\Delta t\Big(\bfF_{x} \bfU^{(k)} + \bfU^{(k)} \bfF_{y}^T\Big),
\end{equation}
where
\begin{equation}
	\label{eq:DIRK_k_RHS}
	\mathbf{W}^{(k-1)} = \bfU^{n} + \Delta t\sum\limits_{\ell=1}^{k-1}{a_{k\ell}\mathbf{Y}_{\ell}}.
\end{equation}

Following the procedure outlined in subsection \ref{sec:implicit_secondorderscheme} yields the Sylvester equations
\begin{equation}\label{eq:Kstep_DIRK_k}
\Big(\mathbf{I}-a_{kk}\Delta t\mathbf{F}_x\Big)\mathbf{K}^{(k)} - \mathbf{K}^{(k)}\Big(a_{kk}\Delta t(\mathbf{F}_y\bfV^{y,(k)}_{\star})^T\bfV^{y,(k)}_{\star}\Big) = \mathbf{W}^{(k-1)}\bfV^{y,(k)}_{\star},
\end{equation}
\begin{equation}\label{eq:Lstep_DIRK_k}
\Big(\mathbf{I}-a_{kk}\Delta t\mathbf{F}_y\Big)\mathbf{L}^{(k)} - \mathbf{L}^{(k)}\Big(a_{kk}\Delta t(\mathbf{F}_x\bfV^{x,(k)}_{\star})^T\bfV^{x,(k)}_{\star}\Big) = (\mathbf{W}^{(k-1)})^T\bfV^{x,(k)}_{\star}.
\end{equation}

Solving equations \eqref{eq:Kstep_DIRK_k}-\eqref{eq:Lstep_DIRK_k} for $\mathbf{K}^{(k)}$ and $\mathbf{L}^{(k)}$ and computing their reduced QR factorizations yield the orthonormal bases, $\bfV^{x,(k)}_{\ddagger}$ and $\bfV^{y,(k)}_{\ddagger}$. Replacing $\bfV^{\cdot,(k)}_{\dagger}$ with $\bfV^{\cdot,(k)}_{\ddagger}$ in the augmented matrices \eqref{eq:aug}, we can reduce according to tolerance $10^{-12}$ to get the pre-truncated intermediate bases $\hat{\bfV}^{\cdot,(k)}$ of rank $\hat{r}^{(k)}$. The Sylvester equation for the S step is then
\begin{align}\label{eq:Sstep_DIRK_k}
\begin{split}
\Big(\mathbf{I} - a_{kk}\Delta t(\hat{\bfV}^{x,(k)})^T\mathbf{F}_x\hat{\bfV}^{x,(k)}\Big)\hat{\mathbf{S}}^{(k)} - &\hat{\mathbf{S}}^{(k)}\Big(a_{kk}\Delta t(\mathbf{F}_y\hat{\bfV}^{y,(k)})^T\hat{\bfV}^{y,(k)}\Big)\\
&\qquad\qquad= (\hat{\bfV}^{x,(k)})^T\mathbf{W}^{(k-1)}\hat{\bfV}^{y,(k)}.
\end{split}
\end{align}

Truncating the factorized solution using an SVD-type procedure according to tolerance $\epsilon$ yields $\bfU^{(k)}=\bfV^{x,(k)}\bfS^{(k)}(\bfV^{y,(k)})^T$. Repeating this process for each subsequent stage in the stiffly accurate DIRK method ultimately computes the updated solution $\mathbf{U}^{(s)}=\mathbf{U}^{n+1}=\bfV^{x,n+1}\bfS^{n+1}(\bfV^{y,n+1})^T$. Butcher tableaus of second- and third-order DIRK methods, as well as the algorithms for the reduced augmentation procedure and implicit RAIL scheme are included in Appendix \ref{app:algos}.

\begin{rem}\label{rem:computcost}
Assuming low-rank structure, that is $r\ll N$, the computational cost of each time-step is dominated by solving the Sylvester equations in the K and L steps; all other matrix decompositions and operations cost at most $\sim N^2r$ flops. The differential operator $\mathbf{I}-a_{kk}\Delta t\mathbf{F}_i$ is size $N\times N$, and $-a_{kk}\Delta t(\mathbf{F}_i\bfV_{\star})^T\bfV_{\star}$ is size $r\times r$. As such, standard Sylvester solvers, e.g., \cite{bartels1972,golub1979hessenberg}, will have a computational cost dominated by $\sim N^3$ flops from computing Householder reflectors for $\mathbf{I}-a_{kk}\Delta t\mathbf{F}_i$. Although not investigated in this paper, the efficiency of the RAIL scheme could be improved by reducing the computational cost of the Sylvester solver. One could diagonalize $\mathbf{I}-a_{kk}\Delta t\mathbf{F}_i$ and solve the transformed Sylvester equations with a reduced cost of $\sim N^2r$ flops; early preliminary results showed this speed-up \cite{nakao2023thesis}. However, this is only advantageous if the differential operator: (1) is not time-dependent, and (2) can be stably diagonalized. Alternatively, one could exploit the possible sparsity of $\mathbf{F}_i$ and use an iterative scheme for which each iteration is only $\sim Nr$ flops.
\end{rem}


\subsection{Analysis of the implicit scheme}\label{sec:analysis}

In this section, we analyze the proposed implicit scheme. Consistency is shown for the first-order RAIL scheme by equating the first-order scheme to the augmented BUG integrator \cite{Ceruti2022a}, and a stability analysis is also provided. The consistency of the high-order RAIL scheme is then discussed without rigorous derivation, but we note that high-order accuracy is observed in all numerical experiments conducted here. We also want to alert the reader of the recent high-order BUG integrators with proven second-order error bounds \cite{Ceruti2024,Kusch2024}; the projected Runge-Kutta methods in \cite{kieri2019projection} also have robust high-order error bounds.

\subsubsection{Consistency}

The proposed first-order implicit scheme is equivalent to the augmented BUG integrator coupled with backward Euler time discretization, with only a couple minor algorithmic differences that do not affect the consistency of the scheme. We do not review the augmented BUG integrator here and refer the reader to Section 3.1 of the original paper \cite{Ceruti2022a}. Namely, the two differences are:
\begin{enumerate}
	\item The augmentation of the updated and current bases is truncated according to a small tolerance of $10^{-12}$. Whereas, the augmented BUG integrator does not have this reduction.
	\item The solution is truncated by first computing the SVD $\hat{\bfS}^{n+1}=\mathbf{A}\mathbf{\Sigma}\mathbf{B}^T$ with $\mathbf{\Sigma}=\text{diag}(\sigma_j)$. The singular values are then truncated according to tolerance $\epsilon$ by choosing the new rank $r^{n+1}$ as the minimal number $r^{n+1}$ such that $\sigma_{r^{n+1}+1}\leq\epsilon$. Whereas, the augmented BUG integrator chooses the new rank $r^{n+1}$ as the minimal number $r^{n+1}$ such that $\sqrt{\sum\limits_{j=r^{n+1}+1}^{2r^n}{\sigma_j^2}}\leq\epsilon$. Simply put, we propose truncating $\hat{\mathbf{S}}^{n+1}$ with respect to the 2-norm; the augmented BUG integrator truncates $\hat{\bfS}^{n+1}$ with respect to the Frobenius norm.
\end{enumerate}

Ceruti et al. \cite{Ceruti2022a} rigorously derived a global error bound for the augmented BUG integrator that is robust to small singular values; the convergence result obviously implies consistency. The proof of which follows the same lines as the proof of the robust error bound shown for the fixed-rank BUG integrator in \cite{Ceruti2022}. Since the analysis done for the BUG integrators does not take round-off errors into account, the two differences mentioned above effectively do not change the overall consistency of the scheme. To be more precise, the reduction of the augmentation with respect to tolerance $1.0e-12$ will add another term to the local error bound on the order of this tolerance, but numerically this contribution is close to machine precision. Hence, the first-order RAIL scheme shares the same consistency result as the augmented BUG integrator coupled with backward Euler. We state the local error bound of the augmented BUG integrator as Theorem \ref{thm:aBUG} but omit the proof.

\begin{thm}[see Theorem 2 in \cite{Ceruti2022a}, and Lemma 4 in \cite{Ceruti2022}]\label{thm:aBUG}
Let $\mathbf{U}(t)$ denote the solution to the matrix differential equation \eqref{eq:MOL_implicit} with initial condition $\mathbf{U}(t^0) = \mathbf{U}^0_e$. Let $\mathbf{F}(t,\mathbf{U}(t))=\mathbf{F}_x\mathbf{U}+\mathbf{U}\mathbf{F}_y^T$. Assume the following conditions hold in the Frobenius norm $\norm{\cdot}=\norm{\cdot}_F$:
\begin{enumerate}
    \item $\mathbf{F}$ is Lipschitz-continuous and bounded: for all $\mathbf{Y},\widetilde{\mathbf{Y}}\in\mathbb{R}^{N\times N}$ and $0\leq t\leq t^1$,
    \begin{equation}
        \norm{\mathbf{F}(t,\mathbf{Y})-\mathbf{F}(t,\widetilde{\mathbf{Y}})}\leq L\norm{\mathbf{Y}-\widetilde{\mathbf{Y}}},\qquad\norm{\mathbf{F}(t,\mathbf{Y})}\leq B.
    \end{equation}
    \item The normal part of $\mathbf{F}(t,\mathbf{Y})$ is $\varepsilon$-small at rank $r^1$ for $\mathbf{Y}$ near $\mathbf{U}(t^1)$ and $t$ near $t^1$. That is, with $\mathbf{P}_{r^1}(\mathbf{Y})$ denoting the orthogonal projection onto the tangent space of the manifold $\mathcal{M}_{r^1}$ of rank-$r^1$ matrices at $\mathbf{Y}\in\mathcal{M}_{r^1}$, it is assumed that
    \begin{equation}
        \norm{(\mathbf{I} - \mathbf{P}_{r^1}(\mathbf{Y}))\mathbf{F}(t,\mathbf{Y})}\leq \varepsilon
    \end{equation}
for all $\mathbf{Y}\in\mathcal{M}_{r^1}$ in a neighborhood of $\mathbf{U}(t^1)$ and $t$ near $t^1$.
\end{enumerate}
Let $\mathbf{U}^1$ denote the low-rank approximation to $\mathbf{U}(t^1)$ obtained after one step of the rank-adaptive integrator with step-size $\Delta t>0$. If $\bfU^0=\bfU^0_e$, then the following local error bound holds:
\begin{equation}
    \norm{\mathbf{U}^1 - \mathbf{U}(t^1)}\leq \Delta t(c_1\varepsilon + c_2\Delta t) + c_3\epsilon,
\end{equation}
where the constants $c_i$ only depend on $L$, $B$, and a bound of the step-size. Here, $\epsilon$ is still the tolerance with which we truncate the updated solution. In particular, the constants are independent of singular values of the exact or approximate solution.
\end{thm}

A rigorous derivation of a local error bound for the high-order RAIL scheme is nontrivial and the topic of future work. We comment on the intuition behind why (and under which assumptions) we expect a high-order robust error bound to exist, but high-order accuracy is only verified in our numerical experiments conducted here. First refer to the Butcher tableau of the second-order DIRK method of interest (see Appendix \ref{app:RK}). Since the first stage is a backward Euler approximation at time $t^{(1)}=t^n+\nu\Delta t$, using the first-order RAIL scheme to obtain $\mathbf{U}^{(1)}$ provides a first-order local truncation error. The idea behind the reduced augmentation is that by including the bases computed at all the previous RK stages in the augmentation, the space we project onto is as rich as possible to allow for the high-order accuracy obtained in the final stage of the RK method in numerical observations. In other words, by including the bases from all previous RK stages, we generate an approximation space as large as possible to allow for higher-order accurate methods. Yet, this likely requires assumptions on the rate of change of the solution due to $\mathbf{F}(t,\mathbf{U}(t))=\mathbf{F}_x\mathbf{U}+\mathbf{U}\mathbf{F}_y^T$ similar to the two assumptions in Theorem \ref{thm:aBUG}. Third- and higher-order DIRK methods are more intricate since the consistency of the method requires balancing many more stages.

\begin{rem}
The augmented bases \eqref{eq:aug} include a first-order approximation at time $t^{(k)}$ computed using the first-order RAIL scheme. To reiterate, the main reason for including this in the augmentation is to include a prediction of the basis at time $t^{(k)}$. Whereas, the reason for including the bases from the previous Runge-Kutta stages is to create a richer space that still enjoys the consistency of the Runge-Kutta method.
\end{rem}

\subsubsection{Stability}

\begin{thm}
Assuming $\mathbf{F}_{x}$ and $\mathbf{F}_{y}$ are symmetric and negative semi-definite, the first-order RAIL scheme using backward Euler is unconditionally stable in $L^2$.
\end{thm}

\begin{proof}
The K and L steps only matter insomuch as providing a good approximation for the sake of consistency and accuracy. They do not affect the stability since the norm of the solution diminishes purely due to the implicit S step. Recalling equation \eqref{eq:Sstep_bE} and dropping the hats for notational ease,
\begin{align*}
    \Big(\mathbf{I} - \Delta t(\bfV^{x,n+1})^T\mathbf{F}_{x}\bfV^{x,n+1}\Big)\bfS^{n+1} - &\bfS^{n+1}\Big(\Delta t(\mathbf{F}_{y}\bfV^{y,n+1})^T\bfV^{y,n+1}\Big)\\
&\qquad\qquad= (\bfV^{x,n+1})^T\bfV^{x,n}\bfS^{n}(\bfV^{y,n})^T\bfV^{y,n+1}. 
\end{align*}
Letting $\mathbf{A}=\Delta t(\bfV^{x,n+1})^T\mathbf{F}_{x}\bfV^{x,n+1}$, $\mathbf{B}=\Delta t(\mathbf{F}_{y}\bfV^{y,n+1})^T\bfV^{y,n+1}$ and\\
$\mathbf{C} = (\bfV^{x,n+1})^T\bfV^{x,n}\bfS^{n}(\bfV^{y,n})^T\bfV^{y,n+1}$, we rewrite equation \eqref{eq:Sstep_bE} as
\begin{equation*}
    (\mathbf{I}-\mathbf{A})\mathbf{S}^{n+1}-\mathbf{S}^{n+1}\mathbf{B} = \mathbf{C}.
\end{equation*}
Note that $\mathbf{A}$ and $\mathbf{B}$ are symmetric if $\mathbf{F}_{x}$ and $\mathbf{F}_{y}$ are symmetric. Performing the diagonalizations $\mathbf{D}=\mathbf{P}^T\mathbf{A}\mathbf{P}$ and $\mathbf{E}=\mathbf{Q}^T\mathbf{B}\mathbf{Q}$, we have
\begin{equation*}
    \mathbf{P}^T(\mathbf{I}-\mathbf{A})\mathbf{P}\mathbf{P}^T\mathbf{S}^{n+1}\mathbf{Q}-\mathbf{P}^T\mathbf{S}^{n+1}\mathbf{Q}\mathbf{Q}^T\mathbf{A}\mathbf{Q}=\mathbf{P}^T\mathbf{C}\mathbf{Q},
\end{equation*}
or rather,
\begin{equation*}
    (\mathbf{I}-\mathbf{D})\mathbf{P}^T\mathbf{S}^{n+1}\mathbf{Q}-\mathbf{P}^T\mathbf{S}^{n+1}\mathbf{Q}\mathbf{E}=\mathbf{P}^T\mathbf{C}\mathbf{Q}.
\end{equation*}
Defining $\check{\mathbf{S}}^{n+1}=\mathbf{P}^T\mathbf{S}^{n+1}\mathbf{Q}$ and $\check{\mathbf{C}}=\mathbf{P}^T\mathbf{C}\mathbf{Q}$, we have
\begin{equation*}
    (1-D_{ii})\check{S}_{ij}^{n+1} - E_{jj}\check{S}_{ij}^{n+1}=\check{C}_{ij},
\end{equation*}
or rather,
\begin{equation*}
    \check{S}_{ij}^{n+1} = \frac{1}{1-D_{ii}-E_{jj}}\check{C}_{ij}.
\end{equation*}
Define the matrix $\mathbf{H}=(\mathbf{V}^{x,n+1})^T\mathbf{F}_{x}\mathbf{V}^{x,n+1}$. Notice that $\mathbf{H}$ is symmetric and negative semi-definite under the assumption that $\mathbf{F}_{x}$ is symmetric and negative semi-definite since
\begin{equation*}
    \mathbf{x}^T\mathbf{H}\mathbf{x} = (\mathbf{V}^{x,n+1}\mathbf{x})^T\mathbf{F}_{x}(\mathbf{V}^{x,n+1}\mathbf{x}) = \mathbf{y}^T\mathbf{F}_{x}\mathbf{y}\leq 0,\qquad\text{for all $\mathbf{x}$}.
\end{equation*}
Thus, $D_{ii}\leq 0$ and $E_{jj}\leq 0$, for all $i$ and $j$. Hence,
\begin{equation*}
    \frac{1}{1-D_{ii}-E_{jj}}\leq 1,\qquad\text{for all $i$ and $j$.}
\end{equation*}
It follows in the vector 2 norm (or the Frobenius norm of the matrix) that
\begin{align*}
    \norm{\mathbf{U}^{n+1}} &= \norm{\mathbf{S}^{n+1}} = \norm{\check{\mathbf{S}}^{n+1}} = \norm{\left(\frac{1}{1-D_{ii}-E_{jj}}\check{C}_{ij}^n\right)_{ij}}\\
    &\leq \norm{\check{\mathbf{C}}^n} = \norm{\mathbf{C}} = \norm{(\mathbf{V}^{x,n+1})^T\mathbf{U}^{n}\mathbf{V}^{y,n+1}} = \norm{\mathbf{P}_{\mathbf{V}^{x,n+1}}\mathbf{P}_{\mathbf{V}^{y,n+1}}\text{vec}(\mathbf{U}^n)}\\
&\leq \norm{\mathbf{U}^{n}}
\end{align*}
which is the desired result of unconditional stability in $L^2$. Note that we let $\mathbf{P}_{\mathbf{X}}$ denote the orthogonal $L^2$ projector onto the subspace spanned by the columns of $\mathbf{X}$ which reduces the norm.
\end{proof}

\section{Extension to implicit-explicit (IMEX) time discretizations}
\label{sec:RAIL_IMEX}

We extend the RAIL integrator to implicit-explicit (IMEX) time discretizations in which nonstiff operators are evolved explicitly and stiff operators are evolved implicitly. IMEX Runge-Kutta schemes from \cite{Ascher1997} are used in this paper. As will be shown, the only algorithmic change is that $\mathbf{W}^{(k-1)}$ in equation \eqref{eq:DIRK_k2} will now include the explicitly evolved terms from the previous stages. Following the same notation as \cite{Ascher1997}, coupling an $s$ stage DIRK scheme with a $\sigma=s+1$ stage explicit Runge-Kutta scheme, with combined order $p$, is denoted as IMEX$(s,\sigma,p)$. These IMEX$(s,\sigma,p)$ schemes are expressed by two Butchers tableaus, one for each of the implicit and explicit methods as seen in Tables \ref{table:IMEX_implicit} and \ref{table:IMEX_explicit}.

\begin{table}[h!]
\begin{minipage}[b]{0.49\linewidth}
\centering
\caption{Implicit Scheme}
\label{table:IMEX_implicit}
\begin{tabular}{c|lllll}
    0&0&0&0&0&0\\
    $c_1$&0&$a_{11}$&0&$\hdots$&0\\
    $c_2$&0&$a_{21}$&$a_{22}$&$\hdots$&0\\
    $\vdots$&$\vdots$&$\vdots$&$\vdots$&$\ddots$&$\vdots$\\
    $c_s$&0&$a_{s1}$&$a_{s2}$&$\hdots$&$a_{ss}$\\
    \hline
    &0&$b_1$&$b_2$&$\hdots$&$b_s$
    \end{tabular}
\end{minipage}
\begin{minipage}[b]{0.49\linewidth}
\centering
\caption{Explicit Scheme}
\label{table:IMEX_explicit}
\begin{tabular}{c|lllll}
    0&0&0&0&$\hdots$&0\\
    $c_1$&$\tilde{a}_{21}$&0&0&$\hdots$&0\\
    $c_2$&$\tilde{a}_{31}$&$\tilde{a}_{32}$&0&$\hdots$&0\\
    $\vdots$&$\vdots$&$\vdots$&$\vdots$&$\ddots$&$\vdots$\\
    $c_s$&$\tilde{a}_{\sigma 1}$&$\tilde{a}_{\sigma 2}$&$\tilde{a}_{\sigma 3}$&$\hdots$&0\\
    \hline
    &$\tilde{b}_1$&$\tilde{b}_2$&$\tilde{b}_3$&$\hdots$&$\tilde{b}_{\sigma}$
    \end{tabular}
\end{minipage}
\end{table}

In Table \ref{table:IMEX_implicit}, the implicit Butcher tableau is padded with a column and row of zeros to cast the DIRK method as an $s+1$ stage scheme to help couple it with the $s+1$ stage explicit scheme. Referring back to the matrix differential equation \eqref{eq:MOL} and letting $\textbf{Im}(t,\bfU)=\mathbf{F}_x\mathbf{U}+\mathbf{U}\mathbf{F}_y^T$ as in equation \eqref{eq:MOL_implicit} for the implicit scheme, the intermediate equation at the $k$th stage is
\begin{align}
\begin{split}
	\label{eq:IMEX_k}
	\bfU^{(k)} &= \bfU^{n} + \Delta t\sum\limits_{\ell=1}^{k}{a_{k\ell}\Big(\mathbf{Y}_{\ell}+\mathbf{\Phi}(t^{(\ell)})\Big)} + \Delta t\sum\limits_{\ell=1}^{k}{\tilde{a}_{k+1,\ell}\tilde{\mathbf{Y}}_{\ell}}\\
	&\equiv \bfU^{n} + \Delta t\sum\limits_{\ell=1}^{k}{a_{k\ell}\Big(\bfF_{x} \bfU^{(\ell)} + \bfU^{(\ell)} \bfF_{y}^T\Big)} +  \Delta t\sum\limits_{\ell=1}^{k}{a_{k\ell}\mathbf{\Phi}(t^{(\ell)})}\\
	&\quad+ \Delta t\sum\limits_{\ell=1}^{k}{\tilde{a}_{k+1,\ell}\textbf{Ex}(t^{(\ell-1)},\bfU^{(\ell-1)})},
\end{split}
\end{align}
where the source term $\mathbf{\Phi}$ is evolved implicitly; the source term could have been evolved explicitly. With all the intermediate stages computed, the final solution is
\begin{equation}
	\label{eq:IMEX_nn}
	\bfU^{n+1} = \bfU^n + \Delta t\sum\limits_{k=1}^{s}{b_k\Big(\mathbf{Y}_k + \mathbf{\Phi}(t^{(k)})\Big)} + \Delta t\sum\limits_{k=1}^{\sigma}{\tilde{b}_k\tilde{\mathbf{Y}}_k}.
\end{equation}

As in the implicit RAIL scheme, we assume stiffly accurate DIRK methods and do not need to compute equation \eqref{eq:IMEX_nn}. Rewriting equation \eqref{eq:IMEX_k} in a more convenient form similar to equation \eqref{eq:DIRK_k2},
\begin{equation}
	\label{eq:IMEX_k2}
	\bfU^{(k)} = \mathbf{W}^{(k-1)} + a_{kk}\Delta t\Big(\bfF_{x} \bfU^{(k)} + \bfU^{(k)} \bfF_{y}^T\Big),
\end{equation}
where
\begin{equation}
	\mathbf{W}^{(k-1)} = \bfU^{n} + \Delta t\sum\limits_{\ell=1}^{k-1}{a_{k\ell}\mathbf{Y}_{\ell}} + \Delta t\sum\limits_{\ell=1}^{k}{a_{k\ell}\mathbf{\Phi}(t^{(\ell)})} + \Delta t\sum\limits_{\ell=1}^{k}{\tilde{a}_{k+1,\ell}\tilde{\mathbf{Y}}_{\ell}}.
\end{equation}

The RAIL algorithm using implicit-explicit time discretizations is practically identical to that using implicit time discretizations. The only changes are $\mathbf{W}^{(k-1)}$ now includes the non-stiff terms at the previous stages, and the approximation $\bfV^{\cdot,(k)}_{\dagger}$ is the first-order solution at time $t^{(k)}$ using IMEX(1,1,1) which is simply a backward Euler-forward Euler discretization.

An error analysis of the RAIL integrator using IMEX time discretizations is more complicated since further assumptions might need to be made on the rate of change of $\mathbf{U}$ due to non-stiff operators. This is a topic of future work. We found that the proposed scheme still observed high-order accuracy in our numerical experiments included in this paper.

\begin{rem}\label{rem:CFLcondition}
The allowable time-stepping sizes for the RAIL scheme using IMEX discretizations obey a CFL condition since an explicit Runge-Kutta scheme is used. If the non-stiff term is smaller or similar in magnitude relative to the stiff terms, numerical instabilities might not appear until larger times. However, to enforce numerical stability when solving advection-diffusion problems with implicit-explicit schemes, the time-stepping size must be upper-bounded by a constant dependent on the ratio of the diffusion and the square of the advection coefficients \cite{wang2015stability}. Although the authors in \cite{wang2015stability} provide upper bounds in the discontinuous Galerkin framework, we observed similar limitations in our numerical tests.
\end{rem}

\begin{rem}\label{rem:construn}
When working with models such as equation \eqref{eq:adv_diff}, conservation is often desired at the discrete level. It is well known that SVD truncation leads to loss of conservation, thus motivating the need for a conservative truncation procedure in the low-rank framework. We note that even when using a standard SVD truncation, we still observed mass conservation on the order of the tolerance $\epsilon$. In our numerical experiments, we utilize the globally mass conservative truncation procedure in \cite{guo2024conservative} inspired by the work in \cite{Einkemmer2021a}. Broadly speaking, we project the solution onto the subspace that preserves the zeroth moment (i.e., mass). This is an orthogonal projection, and the part of the solution in the orthogonal complement contains zero mass and can therefore be truncated without destroying conservation. Although we only focus on conserving mass/number density, higher moments can be similarly conserved with the Local Macroscopic Conservative (LoMaC) truncation procedure in \cite{guo2022local}. This globally mass conservative truncation procedure requires scaling the solution/distribution function by a separable weight function, which we denote by $w(x,y)=w_1(x)w_2(y)$. Furthermore, we note that the final truncated bases must be orthonormalized for our scheme since we assume orthonormal bases. We include the algorithm for the globally mass conservative truncation procedure in the Appendix \ref{app:algos}.
\end{rem}

\section{Numerical tests}
\label{sec:numerics}

In this section, we present results applying the RAIL algorithm to various benchmark problems. Error plots ($L^1$) demonstrate the high-order accuracy in time, rank plots show that the scheme captures low-rank structures of solutions, and relative mass/number density plots show good global conservation. We assume a uniform mesh in space with $N$ gridpoints in each dimension. The spatial derivatives are discretized using spectral collocation methods (assuming periodic boundary conditions); these differentiation matrices can be found in \cite{trefethen2000spectral}. Other differentiation matrices could be used for different boundary conditions. Since the error from the spatial discretization is spectrally accurate, the temporal error will dominate. The singular value tolerance $\epsilon$ is set between $10^{-8}$ and $10^{-6}$, and we define the time-stepping size to be $\Delta t=\lambda\Delta x$ for varying $\lambda>0$. The first $r^0\ll N$ singular vectors/values of the SVD of the initial condition initialize the scheme. Depending on how much we expect the rank to instantaneously increase, the initial rank $r^0$ needs to be large enough to provide a rich enough space to project onto during the first time-step. Lastly, unless otherwise stated, we assume the weight function used to scale the solution in the globally mass conservative truncation procedure is $\mathbf{w}=\mathbf{1}$.


\subsection{The diffusion equation}

\begin{equation}\label{eq:test_diffusion}
    u_t = d_1u_{xx} + d_2u_{yy},\qquad x,y\in(0,14)
\end{equation}
with anisotropic diffusion coefficients $d_1=1/4$ and $d_2=1/9$. We consider the rank-two initial condition
\begin{equation}\label{eq:diffusion_IC}
	u(x,y,t=0) = 0.8\text{exp}(-15((x-6.5)^2+(y-6.5)^2)) + 0.5\text{exp}(-15((x-7.5)^2+(y-7)^2)).
\end{equation}
The domain is made large enough so that the solution's smoothness at the boundary is sufficient to use spectrally accurate differentiation matrices assuming periodic boundary conditions. Figure \ref{fig:tests1and2} shows the error when using the first-, second- and third-order RAIL schemes with backward Euler, DIRK2 and DIRK3, respectively. We used a mesh $N=200$, tolerance $\epsilon=1.0E-08$, final time $T_f=0.5$, $\lambda$ varying from 0.1 to 6, and initial rank $r^0=20$. A full-rank reference solution was computed using a mesh $N=600$, time-stepping size $\Delta t=0.05\Delta x$, and DIRK3. As seen in Figure \ref{fig:tests1and2}, we observe the expected accuracies.

\begin{figure}[t!]
\begin{minipage}[b!]{0.32\linewidth}
	\centering
	\includegraphics[width=\textwidth]{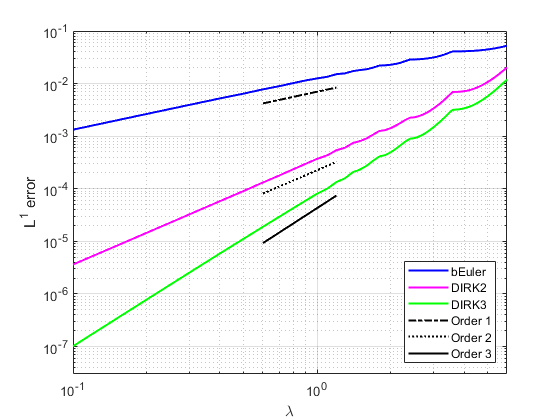}
\end{minipage}
\begin{minipage}[b!]{0.32\linewidth}
	\centering
	\includegraphics[width=\textwidth]{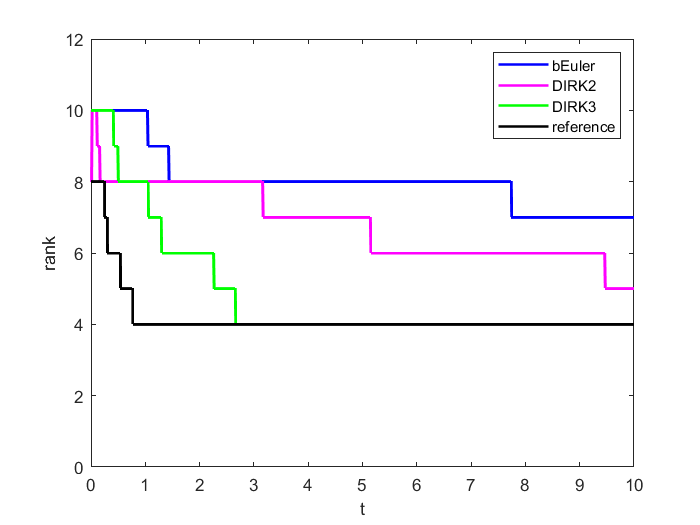}
\end{minipage}
\begin{minipage}[b!]{0.32\linewidth}
	\centering
	\includegraphics[width=\textwidth]{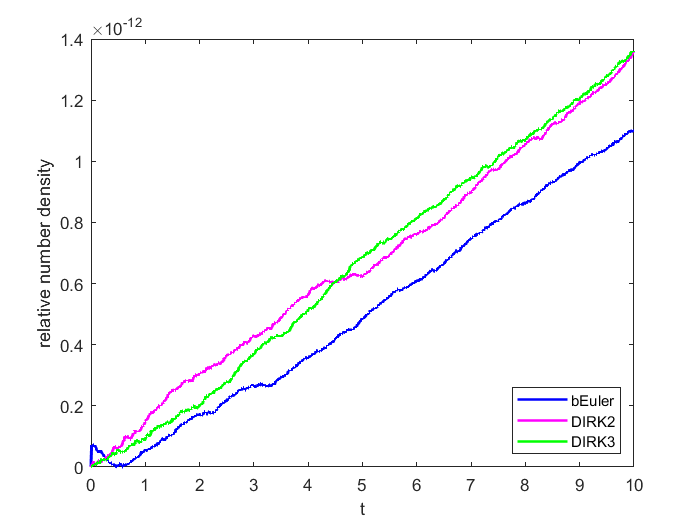}
\end{minipage}
\caption{(left) Error plot for \eqref{eq:test_diffusion} with initial condition \eqref{eq:diffusion_IC} using the first-, second- and third-order RAIL schemes; mesh size $N=200$, tolerance $\epsilon=1.0E-08$, final time $T_f=0.5$, initial rank $r^0=20$. The rank of the solution (middle) and relative number density (right) to \eqref{eq:test_diffusion} with initial condition \eqref{eq:diffusion_IC}; mesh size $N=400$, tolerance $\epsilon=1.0E-08$, time-stepping size $\Delta t=0.3\Delta x$, initial rank $r^0=40$.}
\label{fig:tests1and2}
\end{figure}

The rank of the solution and the relative number number density/mass are shown in Figure \ref{fig:tests1and2}. For the results in those figures, we use a mesh $N=400$, tolerance $\epsilon=1.0E-08$, time-stepping size $\Delta t=0.3\Delta x$, and initial rank $r^0=40$. We also include a reference solution computed with a mesh $N=600$, tolerance $\epsilon=1.0E-08$, time-stepping size $\Delta t = 0.1\Delta x$, initial rank $r^0=60$, and using DIRK3. We expect the rank-two initial condition \eqref{eq:diffusion_IC} to instantaneously increase in rank due to the diffusive dynamics. As seen in Figure \ref{fig:tests1and2}, the RAIL scheme captures this rank increase, with the higher-order schemes better capturing the rank decay towards equilibrium. As one would expect, smaller time-stepping sizes and higher-order time discretizations do a better job at capturing the rank, as seen based on the reference solution. Furthermore, the mass/number density is well-conserved, with the slight changes in mass coming from computing the QR factorization and SVD required to orthogonalize and diagonalize the final solution after truncation.


\subsection{Rigid body rotation with diffusion}

\begin{equation}\label{eq:test_rigid}
    u_t - yu_x + xu_y = d(u_{xx} + u_{yy}) + \phi,\qquad x,y\in(-2\pi,2\pi)
\end{equation}
where $\phi(x,y,t) = (6d - 4xy - 4d(x^2+9y^2))\text{exp}(-(x^2+3y^2+2d t))$, the exact solution is $u(x,y,t) = \text{exp}(-(x^2+3y^2+2d t))$, and $d=1/5$. As seen in Figure \ref{fig:tests3and4}, we observe the expected accuracies when using the first-, second- and third-order RAIL schemes with IMEX(1,1,1), IMEX(2,2,2) and IMEX(4,4,3), respectively. Notice that the high-order accuracy was still achieved despite the diffusion being smaller in magnitude than the advection. We used a mesh $N=200$, tolerance $\epsilon=1.0E-08$, final time $T_f=0.5$, $\lambda$ varying from 0.1 to 2, and initial rank $r^0=20$.

\begin{figure}[t!]
\begin{minipage}[b!]{0.32\linewidth}
	\centering
	\includegraphics[width=\textwidth]{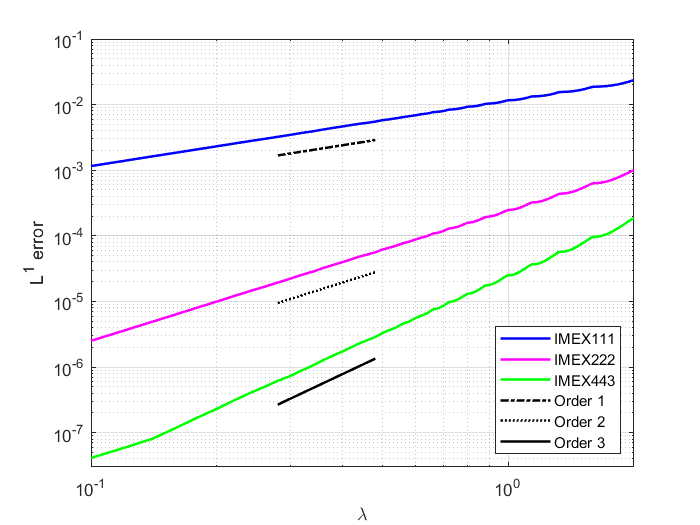}
\end{minipage}
\begin{minipage}[b!]{0.32\linewidth}
	\centering
	\includegraphics[width=\textwidth]{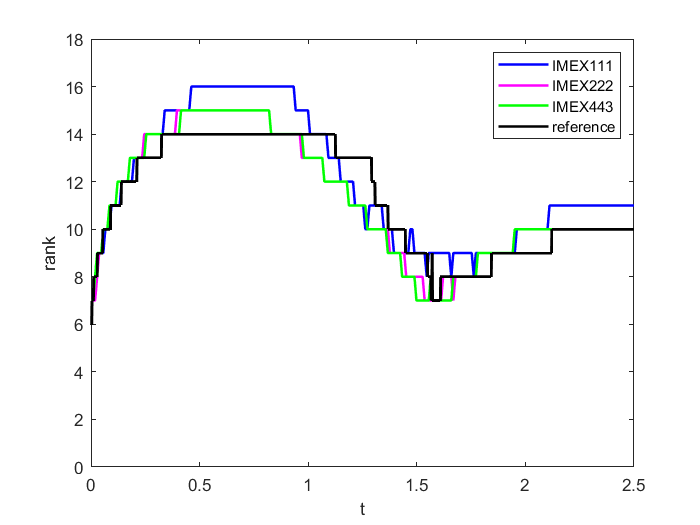}
\end{minipage}
\begin{minipage}[b!]{0.32\linewidth}
	\centering
	\includegraphics[width=\textwidth]{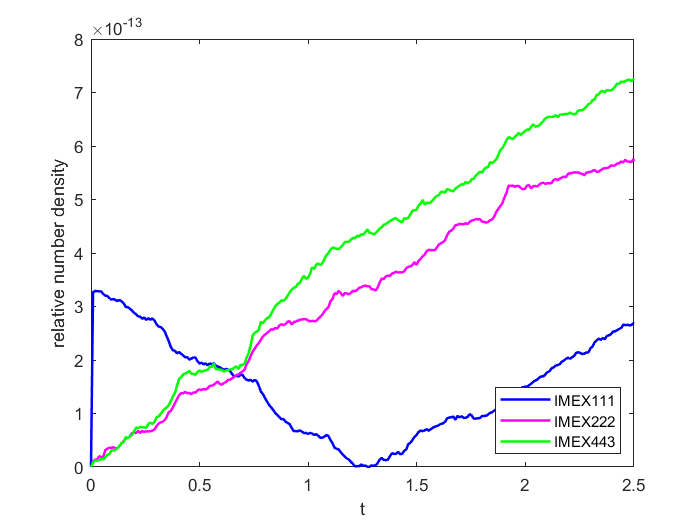}
\end{minipage}
\caption{(left) Error plot for \eqref{eq:test_rigid} with initial condition $\text{exp}(-(x^2+3y^2+2d t))$ using the first-, second- and third-order RAIL schemes; mesh size $N=200$, tolerance $\epsilon=1.0E-08$, final time $T_f=0.5$, initial rank $r^0=20$. The rank of the solution (middle) and relative number density (right) to \eqref{eq:test_rigid} with initial condition $\text{exp}(-(x^2+9y^2))$; mesh size $N=200$, tolerance $\epsilon=1.0E-08$, time-stepping size $\Delta t=0.15\Delta x$, initial rank $r^0=20$.}
\label{fig:tests3and4}
\end{figure}

Setting $\phi(x,y,t)=0$ and $u(x,y,t=0)=\text{exp}(-(x^2+9y^2))$, we expect the solution to rotate counterclockwise about the origin while slowly diffusing. Theoretically, the exact rank should be one at $t=0$ and $t=\pi/2$; the rank should be higher in-between these time stamps. The rank of the solution and the relative number number density/mass are shown in Figure \ref{fig:tests3and4}. For the results in those figures, we use a mesh $N=200$, tolerance $\epsilon=1.0E-08$, time-stepping size $\Delta t=0.15\Delta x$, and initial rank $r^0=20$. We also include a reference solution computed with a mesh $N=400$, tolerance $\epsilon=1.0E-08$, time-stepping size $\Delta t = 0.05\Delta x$, initial rank $r^0=40$, and using IMEX(4,4,3). As seen in Figure \ref{fig:tests3and4}, the RAIL scheme captures this rank behavior for the first-, second- and third-order schemes. Figure \ref{fig:test4_solution} plots the solution at times $t=0$, $t=\pi/4$ and $t=\pi/2$ to better interpret this rank behavior. Lastly, the mass/number density is well-conserved, with the slight changes in mass coming from computing the QR factorization and SVD required to orthogonalize and diagonalize the final solution after truncation.

\begin{figure}[t!]
\begin{minipage}[b!]{0.32\linewidth}
	\centering
	\includegraphics[width=\textwidth]{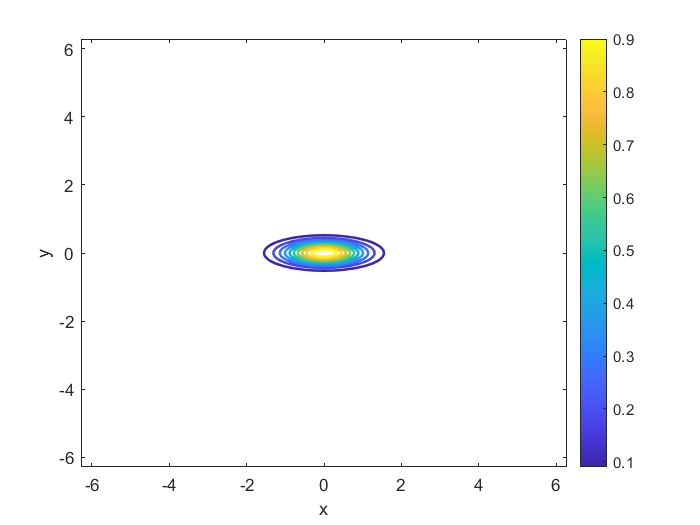}
\end{minipage}
\begin{minipage}[b!]{0.32\linewidth}
	\centering
	\includegraphics[width=\textwidth]{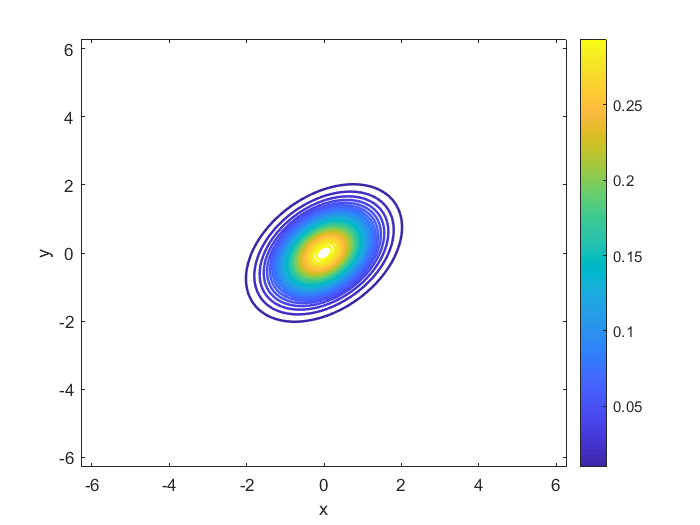}
\end{minipage}
\begin{minipage}[b!]{0.32\linewidth}
	\centering
	\includegraphics[width=\textwidth]{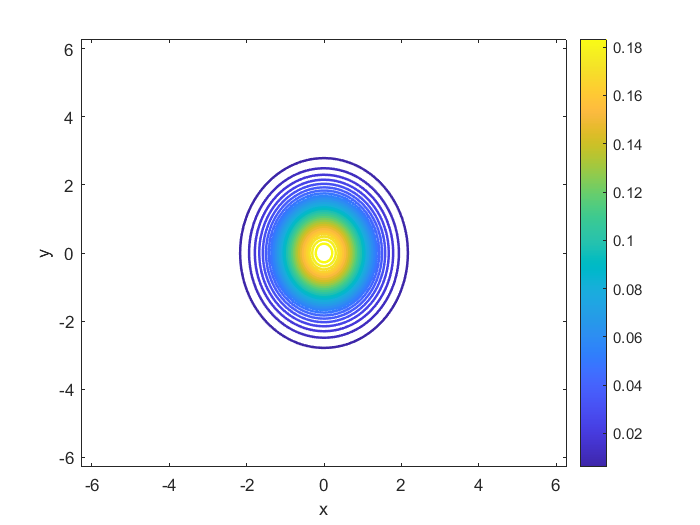}
\end{minipage}
\caption{Various snapshots of the numerical solution to equation \eqref{eq:test_rigid} with initial condition $\text{exp}(-(x^2+9y^2))$. Mesh size $N=200$, tolerance $\epsilon=1.0E-08$, time-stepping size $\Delta t=0.15\Delta x$, initial rank $r^0=20$, using IMEX(4,4,3). Times: 0, $\pi/4$, $\pi/2$.}
\label{fig:test4_solution}
\end{figure}


\subsection{Swirling deformation with diffusion}

\begin{equation}\label{eq:test_swirl}
    u_t - \big(\cos^2{(x/2)}\sin{(y)}f(t)u\big)_x + \big(\sin{(x)}\cos^2{(y/2)}f(t)u\big)_y = u_{xx} + u_{yy},\ x,y\in(-\pi,\pi)
\end{equation}
where we set $f(t) = \cos{(\pi t/T_f)}\pi$. The initial condition is the smooth (with $C^5$ smoothness) cosine bell
\begin{equation}\label{eq:cosbell}
u(x,y,t=0) = 
\begin{cases}
    r_0^b\cos^6{\left(\frac{r^b(x,y)\pi}{2r_0^b}\right)},&\text{if}\ r^b(x,y)<r_0^b,\\
    0,&\text{otherwise},
\end{cases}
\end{equation}
where $r_0^b = 0.3\pi$ and $r^b(x,y) = \sqrt{(x-x_0^b)^2 + (y-y_0^b)^2}$ with $(x_0^b,y_0^b) = (0.3\pi,0)$. Since there is no analytic solution, we use a full-rank reference solution computed with mesh $N=300$, time-stepping size $\Delta t = 0.05\Delta x$, and IMEX(4,4,3). As seen in Figure \ref{fig:tests5and6}, we observe the expected accuracies when using the first-, second- and third-order RAIL schemes with IMEX(1,1,1), IMEX(2,2,2) and IMEX(4,4,3), respectively. We used a mesh $N=100$, tolerance $\epsilon=1.0E-08$, final time $T_f=0.5$, $\lambda$ varying from 0.1 to 2, and initial rank $r^0=15$.

\begin{figure}[t!]
\begin{minipage}[b!]{0.32\linewidth}
	\centering
	\includegraphics[width=\textwidth]{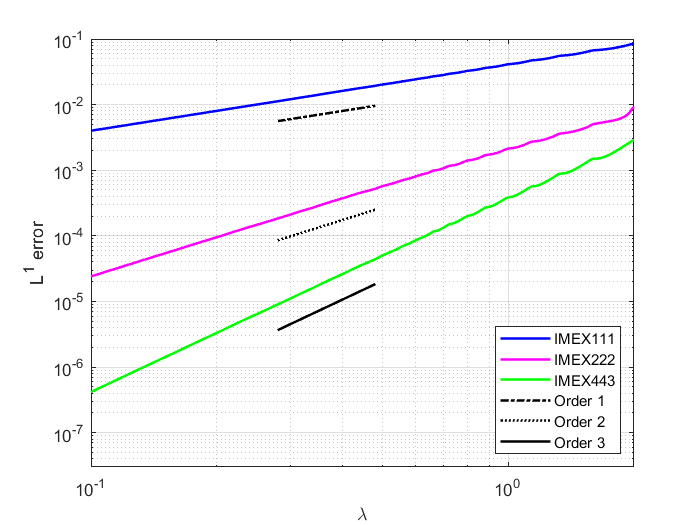}
\end{minipage}
\begin{minipage}[b!]{0.32\linewidth}
	\centering
	\includegraphics[width=\textwidth]{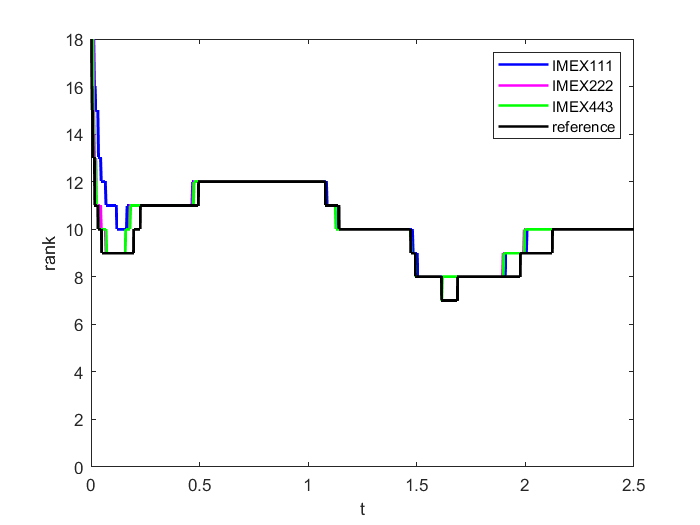}
\end{minipage}
\begin{minipage}[b!]{0.32\linewidth}
	\centering
	\includegraphics[width=\textwidth]{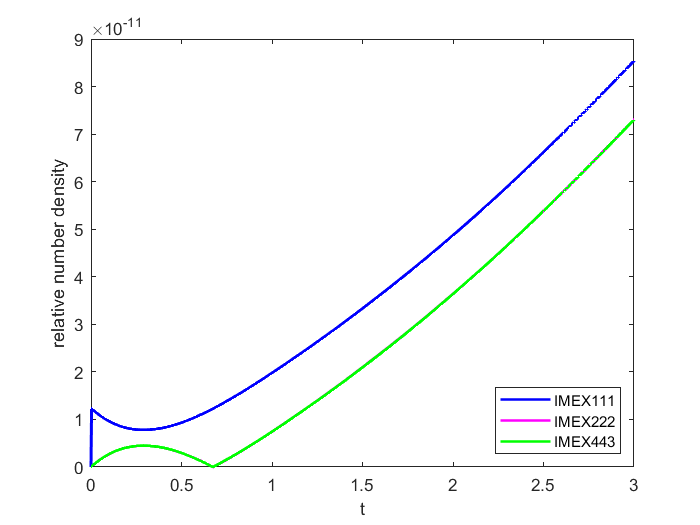}
\end{minipage}
\caption{(left) Error plot for \eqref{eq:test_swirl} with initial condition \eqref{eq:cosbell} using the first-, second- and third-order RAIL schemes; mesh size $N=100$, tolerance $\epsilon=1.0E-08$, final time $T_f=0.5$, initial rank $r^0=15$. The rank of the solution (middle) and relative number density (right) to \eqref{eq:test_swirl} with initial condition \eqref{eq:cosbell}; mesh size $N=300$, tolerance $\epsilon=1.0E-08$, time-stepping size $\Delta t=0.15\Delta x$, initial rank $r^0=30$.}
\label{fig:tests5and6}
\end{figure}

Looking at the flow field, we expect the rank to be roughly the same at times $t=0$, $t=T_f/2$ and $t=T_f$, and to be higher in-between these time stamps. The rank of the solution and the relative number number density/mass are shown in Figure \ref{fig:tests5and6}. For the results in those figures, we use a mesh $N=300$, tolerance $\epsilon=1.0E-08$, time-stepping size $\Delta t=0.15\Delta x$, and initial rank $r^0=30$. We also include a reference solution computed with a mesh $N=300$, tolerance $\epsilon=1.0E-08$, time-stepping size $\Delta t = 0.05\Delta x$, initial rank $r^0=30$, and using IMEX(4,4,3). As seen in Figure \ref{fig:tests5and6}, the RAIL scheme was able to capture the first hump and somewhat the second hump. This might be due to the diffusion starting to have a larger influence on the solution around time $t=2.5$. Furthermore, the mass/number density is globally conserved on the order of $1.0E-11$.


\subsection{0D2V Lenard-Bernstein-Fokker-Planck equation}

We consider a simplified Fokker-Planck model, the Lenard-Bernstein-Fokker-Planck equation \cite{dougherty1964,lenard1958plasma} (also known as the Dougherty-Fokker-Planck equation), used to describe weakly coupled collisional plasmas. Restricting ourselves to a single species plasma in 0D2V phase space, that is, zero spatial dimensions and two velocity dimensions, we solve
\begin{equation}\label{eq:test_LBFP}
    f_t - \big((v_x - \overline{v}_x)f\big)_{v_x} - \big((v_y - \overline{v}_y)f\big)_{v_y} = D(f_{v_xv_x} + f_{v_yv_y}),\quad v_x,v_y\in(-8,8)
\end{equation}
with gas constant $R=1/6$, ion temperature $T=3$, thermal velocity $v_{th}=\sqrt{2RT}=\sqrt{2D}=1$, number density $n=\pi$, and bulk velocities $\overline{v}_x=\overline{v}_y=0$. These quantities were chosen for scaling convenience. The equilibrium solution is the Maxwellian distribution function
\begin{equation}\label{eq:maxwellian}
    f_M(v_x,v_y) = \frac{n}{2\pi RT}\text{exp}\left(-\frac{(v_x-\overline{v}_x)^2 + (v_y-\overline{v}_y)^2}{2RT}\right).
\end{equation}
Relaxation of the system is tested using the initial distribution function
\[f(v_x,v_y,t=0) = f_{M1}(v_x,v_y) + f_{M2}(v_x,v_y),\]
that is, the sum of two randomly generated Maxwellians such that the total macro-parameters are preserved. The number densities, bulk velocities, and temperatures of the Maxwellians are listed in Table \ref{table:twoMaxwellians}. We set $\overline{v}_y=0$ so that the two generated Maxwellians are only shifted along the $v_x$ axis.

\begin{table}[h!] 
\begin{center}
\label{}
\begin{tabular}{|c|c|c|} 
\hline 
&$f_{M1}$&$f_{M2}$\\
\hline
$n$&1.990964530353041&1.150628123236752\\
\hline
$\overline{v}_x$&0.4979792385268875&-0.8616676237412346\\
\hline
$\overline{v}_y$&0&0\\
\hline
$T$&2.46518981703837&0.4107062104302872\\
\hline
\end{tabular} 
\caption{$n=\pi$, $\overline{\mathbf{v}}=\mathbf{0}$, and $T=3$.}
\label{table:twoMaxwellians}
\end{center} 
\end{table}

We use a mesh $N=300$, time-stepping size $\Delta t=0.15\Delta x$, singular value tolerance $\epsilon=1.0E-06$, and initial rank $r^0=30$. We also include a reference solution computed with a mesh $N=400$, tolerance $\epsilon=1.0E-06$, time-stepping size $\Delta t = 0.05\Delta x$, initial rank $r^0=40$, and using IMEX(4,4,3). Unlike the previous examples, the equilibrium solution to equation \eqref{eq:test_LBFP} is a Maxwellian distribution function. To better respect the physics driving the solution to the Maxwellian \eqref{eq:maxwellian} as $t\rightarrow\infty$, we use a Maxwellian distribution for the weight function $w(v_x,v_y)$ used to scale the solution in the truncation procedure. We use the weight function
\begin{equation}\label{eq:robustnessparameter}
    w(v_x,v_y)=w_1(v_x)w_2(v_y)\coloneqq\big(\text{exp}(-v_x^2/2)+\delta\big)\big(\text{exp}(-v_y^2/2)+\delta\big),
\end{equation}
where $\delta=5.0E-09$ is some small robustness parameter to avoid dividing by zero near the boundary. The solution rank and relative number density are shown in Figure \ref{fig:test8}. Similar to the diffusion equation with initial condition \eqref{eq:diffusion_IC}, we expect the rank-two initial distribution function to immediately increase in rank before decaying towards equilibrium. As seen in Figure \ref{fig:test8}, the RAIL scheme captures this behavior, and the number density is globally conserved on the order of $1.0E-11$. Furthermore, Figure \ref{fig:test8_solution} plots the solution at times $t=0$, $t=0.25$ and $t=1$ to better interpret this rank behavior. Lastly, the RAIL scheme captures the theoretical linear exponential decay of the solution towards relaxation, that is, $\norm{f-f_M}_1$. As seen in Figure \ref{fig:test8}, $\norm{f-f_M}_1$ decays linearly down to $\mathcal{O}(1.0E-10)$, despite the tolerance being $\epsilon=1.0E-06$. This indicates that for this problem, the RAIL scheme preserves the equilibrium solution very well.

\begin{figure}[t!]
\begin{minipage}[b!]{0.32\linewidth}
	\centering
	\includegraphics[width=\textwidth]{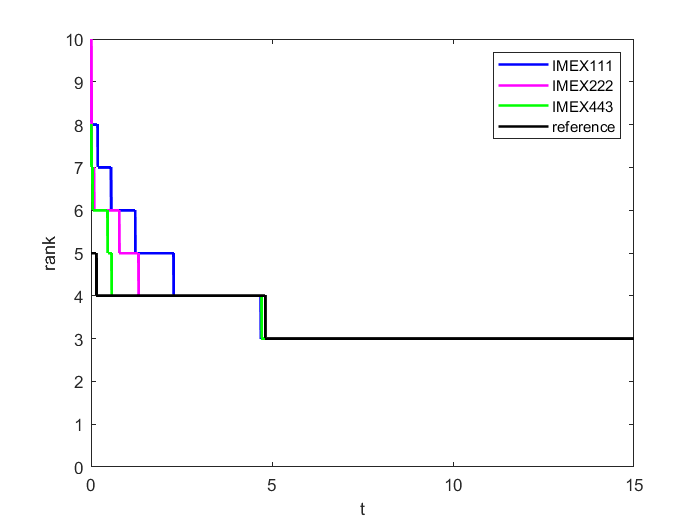}
\end{minipage}
\begin{minipage}[b!]{0.32\linewidth}
	\centering
	\includegraphics[width=\textwidth]{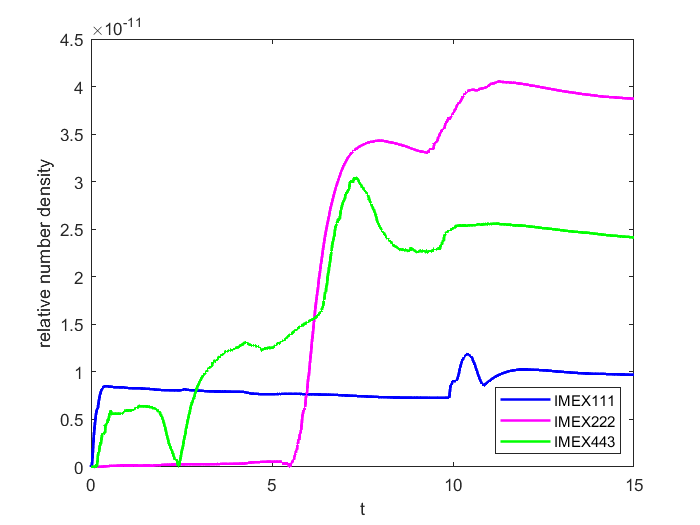}
\end{minipage}
\begin{minipage}[b!]{0.32\linewidth}
	\centering
	\includegraphics[width=\textwidth]{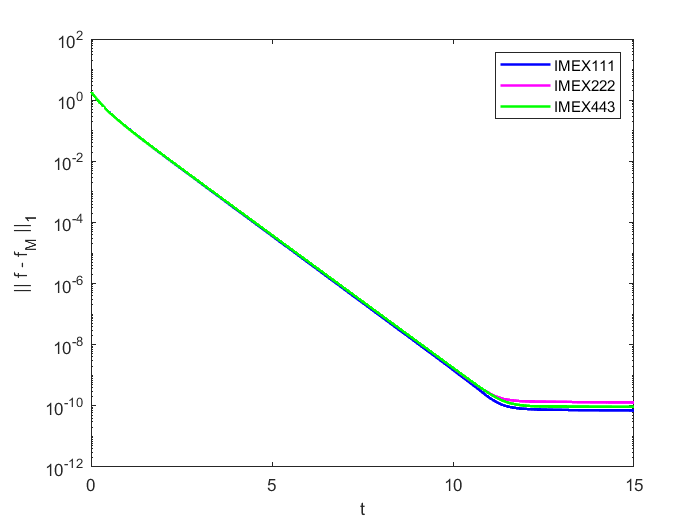}
\end{minipage}
\caption{The rank of the solution (left) and relative number density (middle) to \eqref{eq:test_LBFP} with initial condition $f_{M1}(v_x,v_y) + f_{M2}(v_x,v_y)$; mesh size $N=300$, tolerance $\epsilon=1.0E-06$, time-stepping size $\Delta t=0.15\Delta x$, initial rank $r^0=30$. (right) Error plot for \eqref{eq:test_LBFP} with initial condition $f_{M1}(v_x,v_y) + f_{M2}(v_x,v_y)$; mesh size $N=300$, tolerance $\epsilon=1.0E-06$, final time $T_f=15$, $\lambda=0.15$.}
\label{fig:test8}
\end{figure}

\begin{figure}[h!]
\begin{minipage}[b!]{0.32\linewidth}
	\centering
	\includegraphics[width=\textwidth]{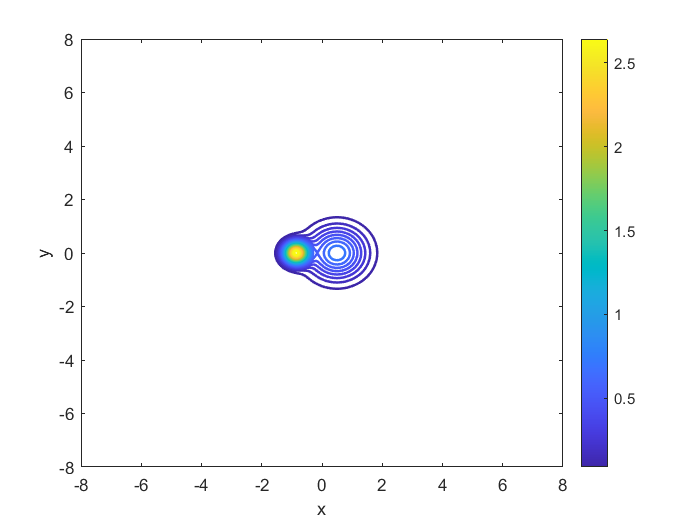}
\end{minipage}
\begin{minipage}[b!]{0.32\linewidth}
	\centering
	\includegraphics[width=\textwidth]{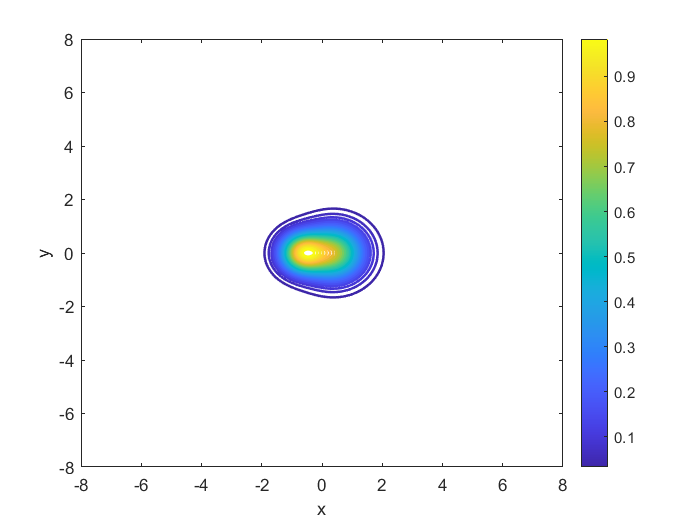}
\end{minipage}
\begin{minipage}[b!]{0.32\linewidth}
	\centering
	\includegraphics[width=\textwidth]{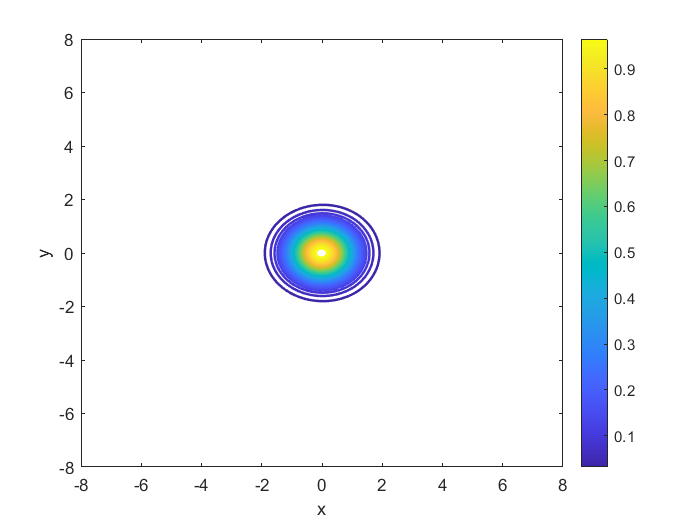}
\end{minipage}
\caption{Various snapshots of the numerical solution to equation \eqref{eq:test_LBFP} with initial condition $f_{M1}(v_x,v_y) + f_{M2}(v_x,v_y)$. Mesh size $N=300$, tolerance $\epsilon=1.0E-06$, time-stepping size $\Delta t=0.15\Delta x$, initial rank $r^0=30$, using IMEX(4,4,3). Times: 0, 0.25, 1.}
\label{fig:test8_solution}
\end{figure}

\section{Conclusion}
\label{sec:conclusion}

In this paper, we proposed a reduced augmentation implicit low-rank (RAIL) integrator for solving advection-diffusion equations and Fokker-Planck models. Both implicit and implicit-explicit RK methods were implemented. The partial differential equations were fully discretized into matrix equations. By spanning the low-order prediction (or updated) bases with the bases from the previous RK stages in a reduced augmentation procedure, the matrix equation was projected onto a richer space. A globally mass conservative truncation procedure was then applied to truncate the solution and keep as low rank as possible. Several tests demonstrated the RAIL scheme's ability to achieve high-order accuracy, capture solution rank, and conserve mass. Ongoing and future work includes deriving rigorous error bounds for the high-order RAIL scheme, as well as extending the algorithm to higher dimensions through other low-rank tensor decompositions.
\section*{Acknowledgements}
\label{sec:acknowledgements}

Research is supported by the National Science Foundation DMS-2111253, Air Force Office of Scientific Research FA955022-1-0390, and Department of Energy DE-SC0023164.

\appendix
\section*{Appendices}

\section{Butcher tables for stiffly accurate DIRK methods \cite{Ascher1997,Hairer1996}}\label{app:RK}
\begin{table}[h!]
    \centering
    \begin{minipage}[b]{0.49\linewidth}
    \centering
    \caption{DIRK2}
    \begin{tabular}{c|ll}
        $\nu$&$\nu$&0\\
        1&$1-\nu$&$\nu$\\
        \hline
        &$1-\nu$&$\nu$
    \end{tabular}
    \ \\
    Let $\nu=1-\sqrt{2}/2$.\\
    \ \\
    \ \\
    \end{minipage}
    \begin{minipage}[b]{0.49\linewidth}
    \centering
    \caption{DIRK3}
    \begin{tabular}{c|lll}
        $\nu$&$\nu$&0&0\\
        $\frac{1+\nu}{2}$&$\frac{1-\nu}{2}$&$\nu$&0\\
        1&$\beta_1$&$\beta_2$&$\nu$\\
        \hline
        &$\beta_1$&$\beta_2$&$\nu$
    \end{tabular}
    \ \\
    Let $\nu\approx 0.435866521508459$,\\
    $\beta_1=-\frac{3}{2}\nu^2+4\nu-\frac{1}{4}$,\\
    and $\beta_2=\frac{3}{2}\nu^2-5\nu+\frac{5}{4}$.
    \end{minipage}
\end{table}

\section{Discretizing advection-diffusion equations with low-rank flow fields and diagonal diffusion tensors}\label{app:advdiffdisc}
In this paper, we are particularly interested in solving advection-diffusion equation \eqref{eq:adv_diff} in two dimensions of the form
\begin{equation}
    \label{eq:adv_diff_2d}
    u_t + \left(a_1(x,y,t)u\right)_x + \left(a_2(x,y,t)u\right)_y = d_1u_{xx} + d_2u_{yy} + \phi(x,y,t),
\end{equation}
where $d_1,d_2>0$, and the flow field $\mathbf{a}(x,y,t)$ and source term $\phi(x,y,t)$ are assumed low-rank. To keep the presentation simple, we assume $a_1(x,y,t)$ and $a_2(x,y,t)$ are both rank-one. A truncated SVD in two dimensions, or a low-rank tensor decomposition in higher dimensions, can be used to extend to flow fields and source terms with (low) ranks greater than one. We show how we discretize linear advection terms below. Assume a rank-one flow field
\begin{equation}\label{eq:flowfield}
    a_{\ell}(x,y,t) = a_{\ell}^x(x,t)a_{\ell}^t(t)a_{\ell}^y(y,t),\qquad \ell=1,2,
\end{equation}
and assume a low-rank source term
\begin{equation}\label{eq:source}
    \phi(x,y,t) = \sum\limits_{p=1}^{r_{\phi}}{\sum\limits_{q=1}^{r_{\phi}}{\phi_{p}^x(x,t)\phi_{pq}^{t}(t)\phi_{q}^y(y,t)}}.
\end{equation}
where $r_{\phi}\ll N$. The basis functions for $\phi(x,y,t)$ do not necessarily need to be orthogonal. The matrix analogues to \eqref{eq:flowfield} and \eqref{eq:source} are respectively
\begin{equation}
    \mathbf{A}_{\ell}(t) = \mathbf{a}_{\ell}^x(t)a_{\ell}^t(t)(\mathbf{a}_{\ell}^y(t))^T\in\mathbb{R}^{N\times N},\qquad \ell=1,2,
\end{equation}
and
\begin{equation}\label{eq:Psi}
    \mathbf{\Phi}(t) = \mathbf{\Phi}^x(t)\mathbf{\Phi}^t(t)(\mathbf{\Phi}^y(t))^T\in\mathbb{R}^{N\times N},
\end{equation}
where $\mathbf{a}_{\ell}^x,\mathbf{a}_{\ell}^y$ ($\ell=1,2$) are column vectors of length $N$, $\mathbf{\Phi}^{x},\mathbf{\Phi}^{y}\in\mathbb{R}^{N\times r_{\phi}}$, and $\mathbf{\Phi}^{t}\in\mathbb{R}^{r_{\phi}\times r_{\phi}}$. Multiplying the rank-one separable flow field \eqref{eq:flowfield} with the low-rank solution \eqref{eq:schmidt}, we see that the matrix form of the flux functions $a_1u$ and $a_2u$ are
\begin{equation}\label{eq:fluxfunctions}
\big(\mathbf{a}_{\ell}^x*\bfV^{x}\big)\big(a_{\ell}^t\bfS\big)\big(\mathbf{a}_{\ell}^y*\bfV^{y}\big)^T,\qquad \ell=1,2,
\end{equation}
where $*$ denotes the columnwise Hadamard product, e.g.,
\begin{equation}
    \mathbf{a}_{\ell}^x*\bfV^{x} \coloneqq \Big[\mathbf{a}_{\ell}^x*\bfV^{x}(\ :\ ,1)\ ,\ \mathbf{a}_{\ell}^x*\bfV^{x}(\ :\ ,2)\ ,\ ...\ ,\ \mathbf{a}_{\ell}^x*\bfV^{x}(\ :\ ,r)\Big].
\end{equation}

As such, the matrix form of the divergence of the flux function is
\begin{equation}
    \big(\mathbf{D}_x(\mathbf{a}_{1}^x*\bfV^{x})\big)\big(a_{1}^t\bfS\big)\big(\mathbf{a}_{1}^y*\bfV^{y}\big)^T + \big(\mathbf{a}_{2}^x*\bfV^{x}\big)\big(a_{2}^t\bfS\big)\big(\mathbf{D}_y(\mathbf{a}_{2}^y*\bfV^{y})\big)^T,
\end{equation}
where $\mathbf{D}_x$ and $\mathbf{D}_y$ are differentiation matrices for the first partial derivatives of $x$ and $y$, respectively. In our numerical tests, we use spectral collocation methods when constructing the differentiation matrices \cite{hesthaven2007spectral,trefethen2000spectral}. Alternatively, other spatial discretizations over local stencils can be used with upwind treatment, e.g., WENO methods \cite{shu2009high}, as has been done in other low-rank frameworks \cite{guo2024conservative,GuoVlasovLoMacDG2023}. We also use spectral collocation methods to construct differentiation matrices $\mathbf{D}_{xx}$ and $\mathbf{D}_{yy}$ for discretizing the diffusion terms. Solving equation \eqref{eq:adv_diff_2d} with the RAIL integrator using IMEX RK methods,
\begin{equation}
	\textbf{Im}(t,\mathbf{U}) = \mathbf{D}_{xx}\mathbf{U} + \mathbf{U}\mathbf{D}_{yy}^T,
\end{equation}
\begin{equation}
	\textbf{Ex}(t,\mathbf{U}) = -\Big\{\big(\mathbf{D}_x(\mathbf{a}_{1}^x*\bfV^{x})\big)\big(a_{1}^t\bfS\big)\big(\mathbf{a}_{1}^y*\bfV^{y}\big)^T + \big(\mathbf{a}_{2}^x*\bfV^{x}\big)\big(a_{2}^t\bfS\big)\big(\mathbf{D}_y(\mathbf{a}_{2}^y*\bfV^{y})\big)^T\Big\}.
\end{equation}

\section{RAIL algorithms and a globally mass conservative truncation procedure}\label{app:algos}
This appendix includes the algorithms for the reduced augmentation procedure (Algorithm \ref{algo:bases}), implicit RAIL scheme (Algorithm \ref{algo:implicit_scheme}), and globally mass conservative truncation procedure \cite{guo2024conservative} (Algorithm \ref{algo:construn}). $\bfV^{\cdot,(k)}_{\ddagger}$ can replace $\bfV^{\cdot,(k)}_{\dagger}$ in Algorithm \ref{algo:bases} to output $\hat{\bfV}^{\cdot,(k)}$. Modifying Algorithm \ref{algo:implicit_scheme} for the implicit-explicit RAIL algorithm is straightforward; see Section \ref{sec:RAIL_IMEX}. We use MATLAB$^{\text{\textregistered}}$ syntax when appropriate. Also, we used the \texttt{sylvester} function in MATLAB$^{\text{\textregistered}}$ to solve all Sylvester equations.
\begin{algorithm}[h!]
	\caption{Reduced augmentation procedure at the $k$th Runge-Kutta stage}
	\label{algo:bases}
		{\bf Input:} $\bfV_{\dagger}^{x,(k)}, \bfV^{x,(k-1)}, ..., \bfV^{x,(1)}, \bfV^{x,(0)}$ and $\bfV_{\dagger}^{y,(k)}, \bfV^{y,(k-1)}, ..., \bfV^{y,(1)}, \bfV^{y,(0)}$.\\
	    {\bf Output:} $\bfV^{x,(k)}_{\star}$ and $\bfV^{y,(k)}_{\star}$.  
	\begin{algorithmic}[1]
	   \State Augment the bases.
        \Statex \texttt{[Qx,Rx] = qr([Vx\textunderscore dagger,Vx\textunderscore\{k-1\},...,Vx\textunderscore 1,Vx\textunderscore 0],0);}
        \Statex \texttt{[Qy,Ry] = qr([Vy\textunderscore dagger,Vy\textunderscore\{k-1\},...,Vy\textunderscore 1,Vy\textunderscore 0],0);}
        \State Reduce with respect to tolerance $1.0e-12$.
        \Statex \texttt{[Vx\textunderscore temp,Sx\textunderscore temp,$\sim$] = svd(Rx,0);}
        \Statex \texttt{[Vy\textunderscore temp,Sy\textunderscore temp,$\sim$] = svd(Ry,0);}
        \Statex \texttt{rx = find(diag(Sx\textunderscore temp)>1.0e-12,1,`last');}
        \Statex \texttt{ry = find(diag(Sy\textunderscore temp)>1.0e-12,1,`last');}
        \Statex \texttt{R = max(rx,ry);}
        \State Keep the corresponding left singular vectors.
        \Statex \texttt{Vx\textunderscore star = Qx*Vx\textunderscore temp(:,1:R);}
        \Statex \texttt{Vy\textunderscore star = Qy*Vy\textunderscore temp(:,1:R);}
	\end{algorithmic}
\end{algorithm}


\begin{algorithm}[t!]
	\caption{The RAIL algorithm for diffusion equations using stiffly accurate DIRK methods}
	\label{algo:implicit_scheme}
		{\bf Input:} $\bfV^{x,n}\bfS^n(\bfV^{y,n})^T$ and $r^n$.\\
	    {\bf Output:} $\bfV^{x,n+1}\bfS^{n+1}(\bfV^{y,n+1})^T$ and $r^{n+1}$. 
	\begin{algorithmic}[1]
        \For{each Runge-Kutta stage $k=1,2,...,s$}
             \State Compute the predictions $\bfV_{\dagger}^{x,(k)}$ and $\bfV_{\dagger}^{y,(k)}$; not needed for first stage.
	      \State Construct the projection bases $\bfV^{x,(k)}_{\star}$ and $\bfV^{y,(k)}_{\star}$ using Algorithm \ref{algo:bases}; not needed in first-order scheme.
            \State Compute $\mathbf{W}^{(k-1)}$ in \eqref{eq:DIRK_k_RHS}.
            \State Solve the $\mathbf{K}$ and $\mathbf{L}$ equations \eqref{eq:Kstep_DIRK_k}-\eqref{eq:Lstep_DIRK_k} for $\mathbf{K}^{(k)}$ and $\mathbf{L}^{(k)}$.
            \State Compute and store the update bases $\bfV_{\ddagger}^{x,(k)}$ and $\bfV_{\ddagger}^{y,(k)}$.
            \Statex\quad\ \texttt{[Vx\textunderscore ddagger,$\sim$] = qr(K,0);}
            \Statex\quad\ \texttt{[Vy\textunderscore ddagger,$\sim$] = qr(L,0);}
            \State Construct the bases $\hat{\bfV}^{x,(k)}$ and $\hat{\bfV}^{y,(k)}$ using Algorithm \ref{algo:bases}.
            \State Using $\hat{\bfV}^{x,(k)}$ and $\hat{\bfV}^{y,(k)}$, solve the $\mathbf{S}$ equation \eqref{eq:Sstep_DIRK_k} and store $\hat{\bfS}^{(k)}$.
	      \State Truncate the solution, and store $\bfV^{x,(k)}$, $\bfS^{(k)}$, $\bfV^{y,(k)}$ and $r^{(k)}$.
		\State Compute and store $\mathbf{Y}_k$ according to equation \eqref{eq:DIRK_k}.
        \EndFor
        \State Store the final solution $\bfV^{x,n+1}$, $\bfS^{n+1}$, $\bfV^{y,n+1}$ and $r^{n+1}$.
	\end{algorithmic}
\end{algorithm}


\begin{algorithm}[t!]
	\caption{A globally mass conservative truncation procedure \cite{guo2024conservative}}
	\label{algo:construn}
		{\bf Input:} Pre-truncated solution $\bfV^{x,(k)}\bfS^{(k)}(\bfV^{y,(k)})^T$ and initial number density $\rho_0$.\\
	    {\bf Output:} (Redefined) truncated solution $\bfV^{x,(k)}\bfS^{(k)}(\bfV^{y,(k)})^T$ and rank $r^{(k)}$. 
	\begin{algorithmic}[1]
        \State Compute $\mathbf{f}_1$.
        \Statex \texttt{Vx\textunderscore f1 = w\textunderscore 1;}
        \Statex \texttt{Vy\textunderscore f1 = w\textunderscore 2;}
        \Statex \texttt{S\textunderscore f1 = rho/((dx*dy)*sum(w\textunderscore 1)*sum(w\textunderscore 2));}
        \Statex \texttt{r\textunderscore f1 = 1;}
        \State Compute and truncate $\mathbf{f}_2$.
        \Statex \texttt{[Qx,Rx] = qr((1./sqrt(w\textunderscore 1)).*[Vx\textunderscore f1,Vx\textunderscore k],0);}
        \Statex \texttt{[Qy,Ry] = qr((1./sqrt(w\textunderscore 2)).*[Vy\textunderscore f1,Vy\textunderscore k],0);}
        \Statex \texttt{[U,S,V] = svd(Rx*blkdiag(-S\textunderscore f1,S\textunderscore k)*Ry',0);}
        \Statex \texttt{r\textunderscore f2 = find(diag(S)>epsilon,1,`last');}
        \Statex \texttt{Vx\textunderscore f2 = sqrt(w\textunderscore 1).*(Qx*U(:,1:r\textunderscore f2));}
        \Statex \texttt{Vy\textunderscore f2 = sqrt(w\textunderscore 2).*(Qy*V(:,1:r\textunderscore f2));}
        \Statex \texttt{S\textunderscore f2 = S(1:r\textunderscore f2,1:r\textunderscore f2);}
        \State Compute the conservatively truncated solution $\mathbf{f}^{n+1}$.
        \Statex \texttt{[Qx,Rx] = qr([Vx\textunderscore f1,Vx\textunderscore f2],0);}
        \Statex \texttt{[Qy,Ry] = qr([Vy\textunderscore f1,Vy\textunderscore f2],0);}
        \Statex \texttt{[U,S,V] = svd(Rx*blkdiag(S\textunderscore f1,S\textunderscore f2)*Ry');}
        \Statex \texttt{Vx\textunderscore k = Qx*U;}
        \Statex \texttt{Vy\textunderscore k = Qy*V;}
        \Statex \texttt{S\textunderscore k = S;}
        \Statex \texttt{r\textunderscore k = r\textunderscore f1 + r\textunderscore f2;}
	\end{algorithmic}
\end{algorithm}

\medskip

\printbibliography

\end{document}